\documentclass[12pt,reqno]{amsart}
\usepackage{amsmath, amsthm, amssymb}

\advance \topmargin by -\headheight
\advance \topmargin by -\headsep
     
\evensidemargin \oddsidemargin
\newtheorem{theorem}{Theorem}[section]
\newtheorem{lemma}[theorem]{Lemma}
\newtheorem{corollary}[theorem]{Corollary}

\theoremstyle{definition}

\theoremstyle{remark}

\numberwithin{equation}{section}

\newcommand{\mmod}[1]{\,\,(\text{mod}\,\,#1)}

\def\bfm{{\mathbf m}}

\def\bfu{{\mathbf u}}
\def\bfv{{\mathbf v}}
\def\bfw{{\mathbf w}}
\def\bfx{{\mathbf x}}
\def\bfy{{\mathbf y}}
\def\bfz{{\mathbf z}}

\def\calB{{\mathcal B}} 
\def\calC{{\mathcal C}}

\def\calJ{{\mathcal J}}

 \def\Ktil{{\widetilde K}}

\def\calR{{\mathcal R}}

\def\Gtil{\widetilde G}\def\Itil{\widetilde I}\def\Ktil{\widetilde K}

\def\dbC{{\mathbb C}}\def\dbN{{\mathbb N}}
\def\dbR{{\mathbb R}}
\def\dbZ{{\mathbb Z}}

\def\grC{{\mathfrak C}}
\def\grd{{\mathfrak d}}
\def\grf{{\mathfrak f}}\def\grF{{\mathfrak F}}
\def\grG{{\mathfrak G}}
\def\grH{{\mathfrak H}}

\def\grm{{\mathfrak m}}\def\grM{{\mathfrak M}}
\def\grS{{\mathfrak S}}
\def\grC{{\mathfrak C}}

\def\grw{{\mathfrak w}}

\def\alp{{\alpha}} \def\bfalp{{\boldsymbol \alpha}}
\def\bet{{\beta}}  \def\bfbet{{\boldsymbol \beta}}
\def\gam{{\gamma}} \def\Gam{{\Gamma}}
\def\del{{\delta}} \def\Del{{\Delta}}
\def\zet{{\zeta}} \def\bfzet{{\boldsymbol \zeta}} 
 
\def\tet{{\theta}} \def\bftet{{\boldsymbol \theta}} 
\def\kap{{\kappa}}
\def\lam{{\lambda}} \def\Lam{{\Lambda}} 

\def\bfxi{{\boldsymbol \xi}}

\def\sig{{\sigma}} \def\Sig{{\Sigma}}

\def\ome{{\omega}} \def\Ome{{\Omega}}
\def\d{{\partial}}
\def\eps{\varepsilon}

\def\le{\leqslant} \def\ge{\geqslant}

\def\d{{\,{\rm d}}}

\def\llbracket{\lbrack\;\!\!\lbrack} \def\rrbracket{\rbrack\;\!\!\rbrack}
\newcommand{\rrbracketsub}[1]{\rrbracket_{\textstyle{_#1}}}

\begin{document}
\title[Vinogradov's mean value theorem]{Multigrade efficient congruencing and Vinogradov's mean value 
theorem}
\author[Trevor D. Wooley]{Trevor D. Wooley}
\address{School of Mathematics, University of Bristol, University Walk, Clifton, Bristol BS8 1TW, United 
Kingdom}
\email{matdw@bristol.ac.uk}
\subjclass[2010]{11L15, 11L07, 11P05, 11P55}
\keywords{Exponential sums, Waring's problem, Hardy-Littlewood method}
\date{}
\begin{abstract} We develop a substantial enhancement of the efficient congruencing method to estimate 
Vinogradov's integral of degree $k$ for moments of order $2s$, thereby obtaining for the first time 
near-optimal estimates for $s>\tfrac{5}{8}k^2$. There are numerous applications. In 
particular, when $k$ is large, the anticipated asymptotic formula in Waring's problem is established for 
sums of $s$ $k$th powers of natural numbers whenever $s\ge 1.543k^2$.\end{abstract}
\maketitle

\section{Introduction} An optimal upper bound in Vinogradov's mean value theorem is now
 known to hold with a number of variables only twice that conjectured to be best possible (see 
\cite[Theorem 1.1]{Woo2012a}). Previous to this very recent advance based on ``efficient congruencing'',
 available technology required that the number of variables be larger by a factor of order $\log k$, for a
 system of degree $k$, a limitation common to all mean value estimates for exponential sums. Hints that
 the conjectured bounds might be proved in full can be glimpsed in a speculative hypothetical
 enhancement to efficient congruencing proposed and heuristically analysed in \cite[\S11]{Woo2012a}.
 Our goal in this paper is to realise an approximation to this enhancement, thereby delivering much of
 what this earlier speculation had promised. In particular, we now come close to establishing an optimal
 upper bound with a number of variables only twenty-five per cent larger than that conjectured to be best
 possible. The central role played by Vinogradov's mean value theorem ensures that applications of our
 new estimates are plentiful.\par

In order to describe our conclusions, we must introduce some notation. When $k$ and $s$ are natural 
numbers, denote by $J_{s,k}(X)$ the number of integral solutions of the Diophantine system
\begin{equation}\label{1.1}
x_1^j+\ldots +x_s^j=y_1^j+\ldots +y_s^j\quad (1\le j\le k),
\end{equation}
with $1\le x_i,y_i\le X$ $(1\le i\le s)$. The {\it main conjecture} in Vinogradov's mean value theorem
 asserts that for each $\eps>0$, one has
\begin{equation}\label{1.2}
J_{s,k}(X)\ll X^\eps (X^s+X^{2s-\frac{1}{2}k(k+1)}).
\end{equation}
Here and throughout this paper, the implicit constants associated with Vinogradov's notation $\ll $ and
 $\gg $ may depend on $s$, $k$ and $\eps$. This conjecture is motivated by the corresponding lower
 bound
\begin{equation}\label{1.3}
J_{s,k}(X)\gg X^s+X^{2s-\frac{1}{2}k(k+1)},
\end{equation}
that arises by considering the diagonal solutions of (\ref{1.1}) with $\bfx=\bfy$, together with a lower
 bound for the product of local densities (see \cite[equation (7.5)]{Vau1997}).\par

We complete the proof of our new estimate for $J_{s,k}(X)$ in \S9. 

\begin{theorem}\label{theorem1.1} Suppose that $k$, $r$ and $s$ are natural numbers with $k\ge 3$,
$$r\le \min \{ k-2,\tfrac{1}{2}(k+1)\}\quad \text{and}\quad s\ge k^2-rk+\tfrac{1}{2}r(r+3)-1.$$
Then for each $\eps>0$, one has
\begin{equation}\label{1.4}
J_{s,k}(X)\ll X^{2s-\frac{1}{2}k(k+1)+\del_s+\eps},
\end{equation}
where $\del_s=\del_{s,k,r}$ is defined by
$$\del_{s,k,r}=\frac{r(r-1)(3k-2r-5)}{6(s-k+1)}.$$
In particular, when $s\ge (k-\tfrac{1}{2}r)^2+\tfrac{1}{4}(r+3)^2$, one has $\del_{s,k,r}<r^2/k$.
\end{theorem}

We refer the reader to Theorem \ref{theorem9.2} for an alternative bound for $J_{s,k}(X)$ which 
is in general slightly more precise than that given by the theorem just announced. Theorem
 \ref{theorem1.1} has the merit of being simpler to state, and also offers slightly sharper bounds in
 situations where $s$ is close to $k^2$. The special case $r=1$ of Theorem \ref{theorem1.1} yields a
 corollary achieving the upper bound (\ref{1.2}) asserted by the main conjecture.

\begin{corollary}\label{corollary1.2}
Suppose that $s$ and $k$ are natural numbers with $k\ge 3$ and $s\ge k^2-k+1$. Then for each
 $\eps>0$, one has $J_{s,k}(X)\ll X^{2s-\frac{1}{2}k(k+1)+\eps}$.
\end{corollary}

This corollary improves on our earlier conclusion \cite[Theorem 1.1]{Woo2013}, in which the 
same upper bound is achieved subject to the constraint $s\ge k^2-1$. Prior to the advent of efficient 
congruencing in \cite{Woo2012a}, meanwhile, estimates of the type supplied by Corollary 
 \ref{corollary1.2} were available only for $s\ge (1+o(1))k^2\log k$ (see \cite{ACK2004}, \cite{Vin1947},
 \cite{Woo1992} and \cite{Woo1996}).\par

The conclusion of Theorem \ref{theorem1.1} improves very significantly on the bounds previously 
available for $J_{s,k}(X)$ in the range $\frac{5}{8}k^2\le s<k^2-1$. By way of comparison, earlier work 
of the author joint with Ford \cite[Theorem 1.2(i)]{FW2013} shows that the bound (\ref{1.4}) holds with 
$\del_s=m^2$ whenever $2m\le k$ and one has $s\ge (k-m)^2+(k-m)$. Thus, in the situation with 
$s=\alp k^2$, in which $\alp$ is a parameter with $\frac{1}{4}\le \alp \le 1$, one has a bound of the
 shape (\ref{1.4}) with $\del_s=(1-\sqrt{\alp})^2k^2+O(k)$. Theorem \ref{theorem1.1}, on the other
 hand, shows that when $\frac{5}{8}\le \alp \le 1$, the estimate (\ref{1.4}) holds with 
$\del_s=C(\alp)k+O(1)$, where
$$C(\alp)=\frac{2-3\alp+(2\alp-1)^{3/2}}{3\alp}.$$
The superiority of our conclusion in the latter interval is clear, since, for the first time, we demonstrate that
 the bound (\ref{1.4}) holds with $\del_s=O(k)$ throughout the interval $\frac{5}{8}k^2<s<k^2$. In 
some sense, therefore, our bounds are near-optimal in the latter range.\par

We pause at this stage to remark that our methods are by no means restricted to the interval 
$\frac{5}{8}k^2<s<k^2$. We have constrained ourselves in this paper to the latter interval in order that 
the ideas underlying our {\it multigrade efficient congruencing} method be transparent. At the same time,
 the new estimates that we make available by imposing this restriction support the bulk of applications 
stemming from this circle of ideas. In forthcoming work \cite{Woo2014}, we tackle the considerable 
technical complications arising from a choice of parameters in which $r$ is permitted to be substantially 
smaller than $\frac{1}{2}(k+1)$. In this way, when $s=\frac{1}{2}k(k+1)$, we are able to establish an 
estimate of the shape (\ref{1.4}) with $\del_s=\tfrac{1}{3}k+o(k)$. The transition to exponents $s$ with 
$s<\frac{1}{2}k(k+1)$ poses further significant challenges. Here, we are able to extend the range 
$1\le s\le \frac{1}{4}(k+1)^2$ in which the estimate $J_{s,k}(X)\ll X^{s+\eps}$ is known to hold. The
 latter, established in \cite[Theorem 1.1]{FW2013}, substantially extends the classical range 
$1\le s\le k+1$ in which the main conjecture (\ref{1.2}) was previously known to hold in the 
diagonally dominated regime.\par

We next explore applications of our methods in the context of Waring's problem. When $s$ and $k$ are
 natural numbers, let $R_{s,k}(n)$ denote the number of representations of the natural number $n$ as 
the sum of $s$ $k$th powers of positive integers. A formal application of the circle method suggests that 
for $k\ge 3$ and $s\ge k+1$, one should have
\begin{equation}\label{1.5}
R_{s,k}(n)=\frac{\Gam(1+1/k)^s}{\Gam(s/k)}\grS_{s,k}(n)n^{s/k-1}+o(n^{s/k-1}),
\end{equation}
where
$$\grS_{s,k}(n)=\sum_{q=1}^\infty\sum^q_{\substack{a=1\\ (a,q)=1}}\Bigl( q^{-1}\sum_{r=1}^q
e(ar^k/q)\Bigr)^se(-na/q).$$
With suitable congruence conditions imposed on $n$, one has $1\ll \grS_{s,k}(n)\ll n^\eps$, so that 
the conjectured relation (\ref{1.5}) constitutes an honest asymptotic formula. Let $\Gtil(k)$ denote the 
least integer $t$ with the property that, for all $s\ge t$, and all sufficiently large natural numbers $n$, one 
has the asymptotic formula (\ref{1.5}). By combining the conclusion of Theorem \ref{theorem1.1} with
 our recent work concerning the asymptotic formula in Waring's problem \cite{Woo2012b}, and the 
enhancement \cite[Theorem 8.5]{FW2013} of Ford's work \cite{For1995}, in \S10 we are able to derive 
new upper bounds for $\Gtil(k)$. We defer to \S10 a full account of these bounds, contenting ourselves
 for the present with the enunciation of the most striking consequences.

\begin{theorem}\label{theorem1.3}
Let $\xi$ denote the real root of the polynomial $20\xi^3+4\xi^2-1$, and put
 $C=(19+75\xi-12\xi^2)/(8+60\xi)$, so that
$$\xi=0.312383\ldots \quad \text{and}\quad C=1.542749\ldots.$$
Then for large values of $k$, one has $\Gtil(k)<Ck^2+O(k)$.
\end{theorem}

Until recently, the sharpest available estimate for $\Gtil(k)$ for larger $k$ was the bound 
$\Gtil(k)\le k^2(\log k+\log \log k+O(1))$ due to Ford \cite{For1995}. This situation was changed with 
the arrival of efficient congruencing, and the most recent work 
\cite[Corollary 9.4]{FW2013} shows that $\Gtil(k)\le 2k^2-2^{2/3}k^{4/3}+O(k)$. Thus the bound
 supplied by Theorem \ref{theorem1.3} provides the first improvement on that of \cite[Theorem 1.4]
{Woo2012a} in which the leading term is reduced by a constant factor. For smaller values of $k$, one may 
compute explicitly the upper bounds for $\Gtil(k)$ that stem from our methods.\par

\begin{theorem}\label{theorem1.4}
With $H(k)$ defined as in Table 1, one has $\Gtil(k)\le H(k)$.
\end{theorem}

$$\boxed{\begin{matrix} k&5&6&7&8&9&10&11&12\\
H(k)&28&43&61&83&107&134&165&199\end{matrix}}$$
$$\boxed{\begin{matrix}k&13&14&15&16&17&18&19&20\\
H(k)&236&276&320&368&418&473&530&592\end{matrix}}$$
\vskip.2cm
\begin{center}\text{Table 1: Upper bounds for $\Gtil(k)$ described in Theorem 
\ref{theorem1.4}.}\end{center}
\vskip.1cm
\noindent 

For comparison, Vaughan \cite[Theorem 1]{Vau1986b} establishes the bound $\Gtil(5)\le 32$, Wooley
 \cite[Corollary 1.7]{Woo2013} gives
 $$\Gtil(6)\le 52,\, \Gtil(7)\le 75,\, \Gtil(8)\le 103,\ \Gtil(9)\le 135,\, \Gtil(10)\le 171,\, \Gtil(11)\le 211,$$
and Ford and Wooley \cite[Corollary 9.3]{FW2013} show that
$$\Gtil(12)\le 253,\, \Gtil(13)\le 299,\ \Gtil(14)\le 349,\, \Gtil(15)\le 403,\ \Gtil(16)\le 460,$$
$$\Gtil(17)\le 521,\ \Gtil(18)\le 587,\, \Gtil(19)\le 656,\ \Gtil(20)\le 729.$$
In particular,  we have in Theorem \ref{theorem1.4} the first improvement on the bound of Vaughan, 
itself closely aligned with that of Hua, for $k=5$. Methods based on Weyl differencing consequently
 remain significant only for $k=3$ and $4$. We note that for $k=4$, in a formal sense our methods
 show that $\Gtil(4)\le 16.311$, falling somewhat short of the bound $\Gtil(4)\le 16$ established by 
Vaughan (see \cite[Theorem 1]{Vau1986b}).\par

We consider further consequences of our new estimates in \S\S11 and 12. In particular, there are 
improvements in estimates of Weyl type for exponential sums, in the distribution of polynomials modulo 
$1$, and in Tarry's problem.\par

We direct the reader to a sketch of the basic efficient congruencing method in \cite[\S2]{Woo2012a} for an 
 introduction to such methods. It may be useful, however, to offer some insight concerning the strategy
 underlying our new multigrade efficient congruencing method. A key step in the efficient congruencing
 approach to Vinogradov's mean value theorem is that of bounding $J_{s,k}(X)$ in terms of an auxiliary
 mean value, in which certain variables are related by the congruences
\begin{equation}\label{1.6}
\sum_{i=1}^kx_i^j\equiv \sum_{i=1}^ky_i^j\mmod{p^{jb}}\quad (1\le j\le k).
\end{equation}
Here, for the purpose of illustration, we suppose that $1\le x_i,y_i\le p^{kb}$, that the $x_i$ are distinct
 modulo $p$, and likewise the $y_i$. The classical approach to Vinogradov's mean value theorem studies
 the situation here with $b=1$. In our first work \cite{Woo2012a} on efficient congruencing, we observe
 that by lifting solutions modulo $p$ to solutions modulo $p^{kb}$ of the system (\ref{1.6}), one may
 suppose without loss that $x_i\equiv y_i\mmod{p^{kb}}$ $(1\le i\le k)$, provided that one inserts a
 factor $k!p^{\frac{1}{2}k(k-1)b}$ into the ensuing estimates to reflect the number of solutions modulo
 $p^{kb}$ for $\bfx$ given a fixed choice of $\bfy$. By applying H\"older's inequality, one obtains a new
 system of the shape (\ref{1.6}) with $b$ replaced by $kb$, and a concentration argument establishes the
 conjectured bound (\ref{1.2}) whenever $s\ge k(k+1)$.\par

The heuristic argument described in \cite[\S11]{Woo2012a} takes as its starting point the conjectural 
proposition that solutions modulo $p$ of the system (\ref{1.6}) may be lifted componentwise, in such a
 manner that one may suppose without loss that $x_i\equiv y_i\mmod{p^{ib}}$ $(1\le i\le k)$, provided
 that one inserts a factor $k!$ into the ensuing estimates. The average degree of the congruence
 concentration is thus essentially halved, greatly improving the efficiency of the method.\par

Lack of independence amongst the variables in such an approach prevents this idea from being anything
 other than one of heuristic significance. However, with $r$ a parameter satisfying $1\le r<k$ to be chosen 
in due course, one may extract from (\ref{1.6}) the congruence relation 
$x_i\equiv y_i\mmod{p^{(k-r)b}}$ $(1\le i\le k)$, at the cost of inserting a factor 
$k!p^{\frac{1}{2}(k-r)(k-r-1)b}$ into the ensuing estimates. By applying H\"older's inequality to the 
associated mean values, one may relate the central mean value to a product of mean values, one in which
 $k-r$ pairs of variables have been extracted subject to a congruence condition modulo $p^{(k-r)b}$, and
 another involving the system of congruences
\begin{equation}\label{1.7}
\sum_{i=1}^rx_i^j\equiv \sum_{i=1}^ry_i^j\mmod{p^{jb}}\quad (1\le j\le k).
\end{equation}
One may preserve the condition that the $x_i$ here are distinct modulo $p$, and likewise the $y_i$. Thus 
we may infer that $x_i\equiv y_i\mmod{p^{(k-r+1)b}}$ $(1\le i\le r)$ at essentially no cost. A further 
application of H\"older's inequality enables us to relate this mean value to another product of mean values, 
one in which a further pair of variables have been extracted subject to a congruence condition modulo 
$p^{(k-r+1)b}$, and another involving a system of the shape (\ref{1.7}), but now with $r$ replaced by 
$r-1$. Repeating this procedure, we successively extract pairs of variables, mutually congruent modulo 
$p^{(k-r+j)b}$ $(1\le j\le r)$, for use in auxiliary mean values elsewhere in the argument. In this way, 
one recovers an approximation to the heuristic basis of our analysis in \cite[\S11]{Woo2012a}. Needless 
to say, there are considerable technical complications both in coaxing this approximation to behave like the 
heuristic approach, and indeed in analysing the consequences only previously discussed in the broadest 
terms.\par

A perusal of \S\S2--9 of this paper will reveal that our multigrade efficient congruencing method is of 
considerable flexibility. The reader may wonder to what extent the particular arrangement of parameters 
employed herein is optimal. Thus, the congruence condition modulo $p^{(k-r)b}$ is applied at the outset, 
and applies to many pairs of variables, with subsequent higher congruences imposed one pair of variables 
at a time. While this arrangement has been pursued following a great deal of time consuming experimentation, 
some guidance is possible for readers seeking to become fully immersed in the underlying methods. In this 
paper we have concentrated on the situation for larger moments, and here as much as possible the full weight 
of congruence savings must be preserved in order to obtain near-optimal estimates. While the initial step of our 
procedure realises the full potential of the congruence condition modulo $p^{(k-r)b}$, subsequent steps 
become possible only following appropriate applications of H\"older's inequality. Each application of the latter 
slightly diminishes the potential savings associated with these subsequent steps, and it seems that for this 
reason, it is more profitable to stack the lower congruence levels ``up front'' rather than spacing out the 
progress to full level $p^{kb}$ more gradually.

\section{Preliminary discussion of infrastructure} We launch our account of the the proof of Theorem 
\ref{theorem1.1}, and the closely allied Theorem \ref{theorem9.2}, by assembling the components
 required for the application of the multigrade efficient congruencing method. Here, where possible, we 
incorporate the simplifying man\oe uvres of \cite{FW2013} into the basic infrastructure developed in 
\cite{Woo2012a} and \cite{Woo2013}. Since we consider the integer $k$ to be fixed, we abbreviate 
$J_{s,k}(X)$ to $J_s(X)$ without further comment. Let $s$ be an arbitrary natural number, and define the
 real number $\lam_s^*$ by means of the relation
$$\lam_s^*=\underset{X\rightarrow \infty}{\lim \sup}\frac{\log J_s(X)}{\log X}.$$
Thus, for each $\eps>0$, and any real number $X$ sufficiently large in terms of $s$, $k$ and $\eps$, one 
has $J_s(X)\ll X^{\lam_s^*+\eps}$. In view of the lower bound (\ref{1.3}), together with a trivial bound 
for $J_s(X)$, we have
\begin{equation}\label{2.1}
\max\{s,2s-\tfrac{1}{2}k(k+1)\}\le \lam_s^*\le 2s,
\end{equation}
while the conjectured upper bound (\ref{1.2}) implies that the first inequality in (\ref{2.1}) should hold 
with equality.\par

We recall some notational conventions from our previous work. The letters $s$ and $k$ denote natural 
numbers with $k\ge 3$, and $\eps$ denotes a sufficiently small positive number. The basic parameter 
occurring in our asymptotic estimates is $X$, a large real number depending at most on $k$, $s$ and 
$\eps$, unless otherwise indicated. Whenever $\eps$ appears in a statement, we assert that the statement
 holds for each $\eps>0$. As usual, we write $\lfloor \psi\rfloor$ to denote the largest integer no larger than 
$\psi$, and $\lceil \psi\rceil$ to denote the least integer no smaller than $\psi$. We make sweeping and 
cavalier use of vector notation. Thus, with $t$ implied from the environment at hand, we write 
$\bfz\equiv \bfw\pmod{p}$ to denote that $z_i\equiv w_i\pmod{p}$ $(1\le i\le t)$, or 
$\bfz\equiv \xi\pmod{p}$ to denote that $z_i\equiv \xi\pmod{p}$ $(1\le i\le t)$, or $[\bfz\mmod{q}]$ to 
denote the $t$-tuple $(\zet_1,\ldots ,\zet_t)$ where for $1\le i\le t$ one has $1\le \zet_i\le q$ and 
$z_i\equiv \zet_i\mmod{q}$. Finally, we employ the convention that whenever 
$G:[0,1)^k\rightarrow \dbC$ is integrable, then
$$\oint G(\bfalp)\d\bfalp =\int_{[0,1)^k}G(\bfalp)\d\bfalp .$$
Thus, on writing
\begin{equation}\label{2.2}
f_k(\bfalp;X)=\sum_{1\le x\le X}e(\alp_1x+\alp_2x^2+\ldots +\alp_kx^k),
\end{equation}
where as usual $e(z)$ denotes $e^{2\pi iz}$, it follows from orthogonality that
\begin{equation}\label{2.3}
J_{s,k}(X)=\oint |f_k(\bfalp;X)|^{2s}\d\bfalp .
\end{equation}

\par We use the index $\iota$ to select choices of parameters appropriate for the proof of our main
 theorems. We take $\iota=0$ to indicate a choice of parameters appropriate for the proof of Theorem 
\ref{theorem9.2}, and $\iota=1$ for a choice appropriate for the proof of Theorem \ref{theorem1.1}. Let 
$r$ be an integral parameter satisfying
\begin{equation}\label{2.4}
1\le r\le \min\{k-2,\tfrac{1}{2}(k+1)\},
\end{equation}
and define
\begin{equation}\label{2.5}
\nu_0(r,s)=\sum_{m=1}^r\frac{m(k-m-1)}{s-m}\quad \text{and}\quad \nu_1(r,s)=0.
\end{equation}
We take $\nu$ to be an integer with $0\le \nu\le \nu_\iota(r,s)$, put
\begin{equation}\label{2.6}
s_\iota(\nu)=k^2-(r+1)k+\tfrac{1}{2}r(r+3)-\nu \quad (\iota=0,1),
\end{equation}
and then consider an integer $s$ satisfying the lower bound $s\ge s_\iota(\nu)$. For brevity we write 
$\grw=s+k-1$ and $\lam=\lam_\grw^*$. Our goal is to establish the upper bound 
$\lam\le 2\grw-\tfrac{1}{2}k(k+1)+\Del$, where $\Del=\Del_\iota (\nu)$ is a carefully chosen target
 exponent satisfying $0\le \Del\le \frac{1}{2}k(k+1)$. We define
\begin{equation}\label{2.7}
\Del_0(\nu)=\sum_{m=1}^r\frac{(m-1)(k-m-1)}{s-m}-\frac{(\nu_0(r,s)-\nu)(r-1)}{s}
\end{equation}
and
\begin{equation}\label{2.8}
\Del_1(\nu)=s^{-1}\sum_{m=1}^r (m-1)(k-m-1).
\end{equation}

\par Let $N$ be an arbitrary natural number, sufficiently large in terms of $s$ and $k$, and put
\begin{equation}\label{2.9}
\tet=(16(s+k))^{-N-1}\quad \text{and}\quad \del=(1000N(s+k))^{-N-1}\tet. 
\end{equation}
In view of the definition of $\lam$, there exists a sequence of natural numbers $(X_l)_{l=1}^\infty$, 
tending to infinity, with the property that
\begin{equation}\label{2.10}
J_\grw(X_l)>X_l^{\lam-\del}\quad (l\in \dbN).
\end{equation}
Also, provided that $X_l$ is sufficiently large, one has the corresponding upper bound
\begin{equation}\label{2.11}
J_\grw(Y)<Y^{\lam+\del}\quad \text{for}\quad Y\ge X_l^{1/2}.
\end{equation}
We now consider a fixed element $X=X_l$ of the sequence $(X_l)_{l=1}^\infty$, which we may assume 
to be sufficiently large in terms of $s$, $k$ and $N$. We put $M=X^\tet$, and note from (\ref{2.9}) that
 $X^\del<M^{1/N}$. Throughout, constants implied in the notation of Landau and Vinogradov may
 depend on $s$, $k$, $N$, and also on $\eps$ in view of our earlier convention, but not on any other
 variable.\par

Let $p$ be a fixed prime number with $M<p\le 2M$ to be chosen in due course. When $c$ and $\xi$ are
 non-negative integers, and $\bfalp \in [0,1)^k$, define
\begin{equation}\label{2.12}
\grf_c(\bfalp;\xi)=\sum_{\substack{1\le x\le X\\ x\equiv \xi\mmod{p^c}}}e(\alp_1x+\alp_2x^2+\ldots
 +\alp_kx^k).
\end{equation}
As in \cite{Woo2012a}, we must consider well-conditioned tuples of integers belonging to distinct
 congruence classes modulo a suitable power of $p$. When $1\le m\le k-1$, denote by $\Xi_c^m(\xi)$
 the set of integral $m$-tuples $(\xi_1,\ldots ,\xi_m)$, with
$$1\le \bfxi\le p^{c+1}\quad \text{and}\quad \bfxi\equiv \xi\pmod{p^c},$$
and satisfying the property that $\xi_i\not \equiv \xi_j\pmod{p^{c+1}}$ for $i\ne j$. We then put
\begin{equation}\label{2.13}
\grF_c^m(\bfalp;\xi)=\sum_{\bfxi\in \Xi_c^m(\xi)}\prod_{i=1}^m\grf_{c+1}(\bfalp;\xi_i),
\end{equation}
where the exponential sums $\grf_{c+1}(\bfalp;\xi_i)$ are defined via (\ref{2.12}).\par

As in our previous work on the efficient congruencing method, certain mixed mean values play a critical 
role within our arguments. When $a$ and $b$ are positive integers, we define
\begin{equation}\label{2.14}
I_{a,b}^m(X;\xi,\eta)=\oint |\grF_a^m(\bfalp;\xi)^2\grf_b(\bfalp;\eta)^{2\grw-2m}|\d\bfalp
\end{equation}
and
\begin{equation}\label{2.15}
K_{a,b}^m(X;\xi,\eta)=\oint |\grF_a^m(\bfalp;\xi)^2\grF_b^{k-1}(\bfalp;\eta)^2
\grf_b(\bfalp;\eta)^{2s-2m}|\d \bfalp .
\end{equation}
We remark that, in order to permit the number of variables subject to the congruencing process to vary, 
which is tantamount to allowing the parameter $m$ to vary likewise, the definition of the mean value 
$K_{a,b}^m(X;\xi,\eta)$ is necessarily more complicated than analogues in our previous work on efficient
 congruencing. This will become apparent in \S6.\par

For future reference, it is useful to note that by orthogonality, the mean value $I_{a,b}^m(X;\xi,\eta)$
 counts the number of integral solutions of the system
\begin{equation}\label{2.16}
\sum_{i=1}^m(x_i^j-y_i^j)=\sum_{l=1}^{\grw-m}(v_l^j-w_l^j)\quad (1\le j\le k),
\end{equation}
with
$$1\le \bfx,\bfy,\bfv,\bfw\le X,\quad \bfv\equiv \bfw\equiv \eta \mmod{p^b},$$
$$[\bfx \mmod{p^{a+1}}]\in \Xi_a^m(\xi)\quad \text{and}\quad [\bfy\mmod{p^{a+1}}]
\in \Xi_a^m(\xi).$$
Similarly, the mean value $K_{a,b}^m(X;\xi,\eta)$ counts the number of integral solutions of the system
\begin{equation}\label{2.17}
\sum_{i=1}^m(x_i^j-y_i^j)=\sum_{l=1}^{k-1}(u_l^j-v_l^j)+\sum_{n=1}^{s-m}(w_n^j-z_n^j)
\quad (1\le j\le k),
\end{equation}
with
$$1\le \bfx,\bfy\le X,\quad [\bfx \mmod{p^{a+1}}]\in \Xi_a^m(\xi),\quad [\bfy\mmod{p^{a+1}}]
\in \Xi_a^m(\xi),$$
$$1\le \bfu,\bfv\le X,\quad [\bfu \mmod{p^{b+1}}]\in \Xi_b^{k-1}(\eta),
\quad [\bfv\mmod{p^{b+1}}]\in 
\Xi_b^{k-1}(\eta),$$
$$1\le \bfw,\bfz\le X,\quad \bfw\equiv \bfz\equiv \eta\mmod{p^b}.$$

Given a solution $\bfx$, $\bfy$, $\bfu$, $\bfv$, $\bfw$, $\bfz$ of the system (\ref{2.17}), an application
 of the Binomial Theorem shows that for $1\le j\le k$, one has
$$\sum_{i=1}^m((x_i-\eta)^j-(y_i-\eta)^j)=\sum_{l=1}^{k-1}
((u_l-\eta)^j-(v_l-\eta)^j)+\sum_{n=1}^{s-m}((w_n-\eta)^j-(z_n-\eta)^j).$$
But in any solution counted by $K_{a,b}^{m}(X;\xi,\eta)$, one has 
$\bfu\equiv \bfv\equiv \eta \mmod{p^b}$ and $\bfw\equiv \bfz\equiv \eta \mmod{p^b}$. We therefore 
deduce that
\begin{equation}\label{2.18}
\sum_{i=1}^m(x_i-\eta)^j\equiv \sum_{i=1}^m(y_i-\eta)^j\mmod{p^{jb}}\quad (1\le j\le k).
\end{equation}

\par It is convenient to put
\begin{equation}\label{2.19}
I_{a,b}^m(X)=\max_{1\le \xi\le p^a}\max_{\substack{1\le \eta\le p^b\\ \eta\not\equiv \xi\mmod{p}}}
I_{a,b}^m(X;\xi,\eta)
\end{equation}
and
\begin{equation}\label{2.20}
K_{a,b}^m(X)=\max_{1\le \xi\le p^a}\max_{\substack{1\le \eta\le p^b\\ \eta\not\equiv \xi\mmod{p}}}
K_{a,b}^m(X;\xi,\eta).
\end{equation}
Note here that although these mean values implicitly depend on our choice of the prime $p$, this choice 
depends on $s$, $k$, $r$, $\tet$ and $X_l$ alone. Since we fix $p$ in the precongruencing step described
 in \S5, following the proof of Lemma \ref{lemma5.1}, the particular choice will ultimately be rendered 
irrelevant.\par

The precongruencing step requires a definition of $K_{0,b}^m(X)$ aligned with the conditioning idea, and
 this we now describe. When $\bfzet$ is a tuple of integers, we denote by 
$\Xi^m(\bfzet)$ the set of $m$-tuples $(\xi_1,\ldots ,\xi_m)\in \Xi_0^m(0)$ such that 
$\xi_i\not\equiv \zet_j\mmod{p}$ for all $i$ and $j$. Recalling (\ref{2.12}), we put
$$\grF^m(\bfalp;\bfzet)=\sum_{\bfxi \in \Xi^m(\bfzet)}\prod_{i=1}^m\grf_1(\bfalp;\xi_i),$$
and then define
$$\Itil_c^m(X;\eta)=\oint |\grF^m(\bfalp;\eta)^2\grf_c(\bfalp;\eta)^{2\grw-2m}|\d\bfalp ,$$
\begin{equation}\label{2.21}
\Ktil_c^m(X;\eta)=\oint |\grF^m(\bfalp;\eta)^2\grF_c^{k-1}(\bfalp;\eta)^2\grf_c(\bfalp;\eta)^{2s-2m}|
\d\bfalp ,
\end{equation}
\begin{equation}\label{2.22}
K_{0,c}^m(X)=\max_{1\le \eta\le p^c}\Ktil_c^m(X;\eta).
\end{equation}

\par As in \cite{Woo2012a}, our arguments are simplified by making transparent the relationship between
 mean values and their anticipated magnitudes, although for present purposes we adopt a more flexible 
notation than that employed earlier. When $\grd$ and $\rho$ are non-negative numbers, we adopt the 
convention that
\begin{equation}\label{2.23}
\llbracket J_\grw(X)\rrbracketsub{\grd}=\frac{J_\grw(X)}{X^{2\grw-\frac{1}{2}k(k+1)+\grd}},
\end{equation}
\begin{equation}\label{2.24}
\llbracket I_{a,b}^m(X)\rrbracketsub{{\grd,\rho}}=\frac{(M^{ka-b})^\rho 
I_{a,b}^m(X)}{(X/M^b)^{2\grw-2m}(X/M^a)^{2m-\frac{1}{2}k(k+1)+\grd}}
\end{equation}
and
\begin{equation}\label{2.25}
\llbracket K_{a,b}^m(X)\rrbracketsub{{\grd,\rho}}=\frac{(M^{ka-b})^\rho 
K_{a,b}^m(X)}{(X/M^b)^{2\grw-2m}(X/M^a)^{2m-\frac{1}{2}k(k+1)+\grd}}.
\end{equation}
Using this notation, the bounds (\ref{2.10}) and (\ref{2.11}) may be rewritten as
\begin{equation}\label{2.26}
\llbracket J_\grw(X)\rrbracketsub{\Del}>X^{\Lam-\del}\quad \text{and}\quad 
\llbracket J_\grw(Y)\rrbracketsub{\Del}<Y^{\Lam+\del}\quad (Y\ge X^{1/2}),
\end{equation}
where $\Lam=\Lam(\Del)$ is defined by
\begin{equation}\label{2.27}
\Lam(\Del)=\lam-2\grw+\tfrac{1}{2}k(k+1)-\Del.
\end{equation}

\par We finish this section by recalling two simple estimates that encapsulate the translation-dilation 
invariance of the Diophantine system (\ref{1.1}).

\begin{lemma}\label{lemma2.1}
Suppose that $c$ is a non-negative integer with $c\tet\le 1$. Then for each natural number $u$, one has
$$\max_{1\le \xi\le p^c}\oint |\grf_c(\bfalp;\xi)|^{2u}\d\bfalp \ll_u J_u(X/M^c).$$
\end{lemma}

\begin{proof} This is \cite[Lemma 3.1]{Woo2012a}.
\end{proof}

\begin{lemma}\label{lemma2.2}
Suppose that $c$ and $d$ are non-negative integers with $c\le \tet^{-1}$ and $d\le \tet^{-1}$. Then 
whenever $u,v\in \dbN$ and $\xi,\zet\in \dbZ$, one has
$$\oint |\grf_c(\bfalp;\xi)^{2u}\grf_d(\bfalp;\zet)^{2v}|\d\bfalp \ll_{u,v}(J_{u+v}(X/M^c))^{u/(u+v)}
(J_{u+v}(X/M^d))^{v/(u+v)}.$$
\end{lemma}

\begin{proof} This is \cite[Corollary 2.2]{FW2013}.
\end{proof}

\section{Auxiliary systems of congruences} There are two primary regimes of interest so far as auxiliary
 congruences are concerned. Fortunately, we are able to extract suitable estimates from our previous work
 \cite{FW2013, Woo2012a, Woo2013}, though this requires that we recall in detail the notation introduced 
in the latter papers. When $a$ and $b$ are integers with $1\le a<b$, we denote by 
$\calB_{a,b}^n(\bfm;\xi,\eta)$ the set of solutions of the system of congruences
\begin{equation}\label{3.1}
\sum_{i=1}^n(z_i-\eta)^j\equiv m_j\mmod{p^{jb}}\quad (1\le j\le k),
\end{equation}
with $1\le \bfz\le p^{kb}$ and $\bfz\equiv \bfxi\pmod{p^{a+1}}$ for some $\bfxi\in \Xi_a^n(\xi)$. We
 define an equivalence relation $\calR(\lam)$ on integral $n$-tuples by declaring $\bfx$ and $\bfy$ to be 
$\calR(\lam)$-equivalent when $\bfx\equiv \bfy\pmod{p^\lam}$. We then write 
$\calC_{a,b}^{n,h}(\bfm;\xi,\eta)$ for the
 set of $\calR(hb)$-equivalence classes of $\calB_{a,b}^n(\bfm;\xi,\eta)$, and we define 
$B_{a,b}^{n,h}(p)$ by putting
\begin{equation}\label{3.2}
B_{a,b}^{n,h}(p)=\max_{1\le \xi\le p^a}
\max_{\substack{1\le \eta\le p^b\\ \eta\not\equiv \xi\mmod{p}}}
\max_{1\le \bfm\le p^{kb}}\text{card}(\calC_{a,b}^{n,h}(\bfm;\xi,\eta)).
\end{equation}

\par When $a=0$ we modify these definitions, so that $\calB_{0,b}^n(\bfm;\xi,\eta)$ denotes the set of
 solutions of the system of congruences (\ref{3.1}) with $1\le \bfz\le p^{kb}$ and 
$\bfz\equiv \bfxi\mmod{p}$ for some $\bfxi\in \Xi_0^n(\xi)$, and for which in addition 
$\bfz\not\equiv \eta\mmod{p}$. As in the situation in which one has $a\ge 1$, we write 
$\calC_{0,b}^{n,h}(\bfm;\xi,\eta)$ for the set of $\calR(hb)$-equivalence classes of 
$\calB_{0,b}^n(\bfm;\xi,\eta)$, but we define $B_{0,b}^{n,h}(p)$ by putting
$$B_{0,b}^{n,h}(p)=\max_{1\le \eta\le p^b}\max_{1\le \bfm\le p^{kb}}
\text{card}(\calC_{0,b}^{n,h}(\bfm;0,\eta)).$$
We note that the choice of $\xi$ in this situation with $a=0$ is irrelevant.\par

The next lemma records the two estimates for $B_{a,b}^{n,h}(p)$ of use in the regimes of interest to us.

\begin{lemma}\label{lemma3.1} Let $a$ and $b$ be integers with $0\le a<b$ and $b\ge (r-1)a$. Then
\begin{equation}\label{3.3}
B_{a,b}^{k-1,k-r}(p)\le k!p^{\mu b+\nu a},
\end{equation}
where
\begin{equation}\label{3.4}
\mu=\tfrac{1}{2}(k-r-1)(k-r-2)\quad \text{and}\quad \nu=\tfrac{1}{2}(k-r-1)(k+r-2).
\end{equation}
In addition, subject to the additional hypothesis $1\le j\le r$, one has
\begin{equation}\label{3.5}
B_{a,b}^{r-j+1,k-r+j}(p)\le k!p^{(r-j)a}.
\end{equation}
\end{lemma}

\begin{proof} We apply \cite[Lemma 3.3]{FW2013}. Thus, provided that $k$, $R$ and $T$ satisfy
\begin{equation}\label{3.6}
k\ge 2,\quad \max\{2,\tfrac{1}{2}(k-1)\}\le T\le k,\quad 1\le R\le k\quad \text{and}\quad R+T\ge k,
\end{equation}
and in addition
\begin{equation}\label{3.7}
0\le a<b\quad \text{and}\quad b\ge (k-T-1)a,
\end{equation}
one has
$$B_{a,b}^{R,T}(p)\le k!p^{\mu'b+\nu'a},$$
where
\begin{equation}\label{3.8}
\mu'=\tfrac{1}{2}(T+R-k)(T+R-k-1)\quad \text{and}\quad \nu'=\tfrac{1}{2}(T+R-k)(k+R-T-1).
\end{equation}
For the first conclusion of the lemma, we take $R=k-1$ and $T=k-r$, noting that the condition (\ref{2.4})
 ensures that $T\ge \max\{ \frac{1}{2}(k-1),2\}$ and $R+T\ge k+1$. Thus, subject to the conditions 
$0\le a<b$ and $b\ge (r-1)a$ imported from (\ref{3.7}), one obtains the bound (\ref{3.3}) by computing
 the exponents $\mu'$ and $\nu'$ given by (\ref{3.8}). For the second conclusion, we take $R=r-j+1$ and 
$T=k-r+j$. Here, the conditions imposed by (\ref{3.6}) are easily verified. In this case, the constraints 
(\ref{3.7}) are satisfied when $1\le j\le r$ provided that $0\le a<b$ and $b\ge (r-1)a$, and so the desired 
conclusion (\ref{3.5}) again follows by evaluating the exponents delivered by (\ref{3.8}).
\end{proof}
 
\section{The conditioning process} The mean value $K_{a,b}^m(X;\xi,\eta)$ differs from the analogue 
previously employed in efficient congruencing methods, and thus we must discuss the conditioning process 
in some detail. Our goal will now be to replace a factor $\grf_b(\bfalp;\eta)^{2k-2}$ occurring in 
(\ref{2.14}) by the conditioned factor $\grF_b^{k-1}(\bfalp;\eta)^2$ in (\ref{2.15}).

\begin{lemma}\label{lemma4.1} Let $a$ and $b$ be integers with $b>a\ge 1$. Then one has
$$I_{a,b}^{k-1}(X)\ll K_{a,b}^{k-1}(X)+M^{2s/3}I_{a,b+1}^{k-1}(X).$$
\end{lemma}

\begin{proof} Consider fixed integers $\xi$ and $\eta$ with $\eta\not\equiv \xi\mmod{p}$. Let $T_1$ 
denote the number of integral solutions $\bfx$, $\bfy$, $\bfv$, $\bfw$ of the system (\ref{2.16}) counted 
by $I_{a,b}^{k-1}(X;\xi,\eta)$ in which $v_1,\ldots,v_s$ together occupy at least $k-1$ distinct residue 
classes modulo $p^{b+1}$, and let $T_2$ denote the corresponding number of solutions in which these 
integers together occupy at most $k-2$ distinct residue classes modulo $p^{b+1}$. Then
\begin{equation}\label{4.1}
I_{a,b}^{k-1}(X;\xi,\eta)=T_1+T_2.
\end{equation}

\par We first estimate $T_1$. Recall the definitions (\ref{2.13}), (\ref{2.14}) and (\ref{2.15}). Then by 
orthogonality and an application of H\"older's inequality, one finds that
\begin{align}
T_1&\le \binom{s}{k-1}\oint |\grF_a^{k-1}(\bfalp;\xi)|^2
\grF_b^{k-1}(\bfalp;\eta)\grf_b(\bfalp;\eta)^{s-k+1}\grf_b(-\bfalp;\eta)^s\d\bfalp \notag \\
&\ll \left( K_{a,b}^{k-1}(X;\xi,\eta)\right)^{1/2}\left( I_{a,b}^{k-1}(X;\xi,\eta)\right)^{1/2}.\label{4.2}
\end{align}
Next we estimate $T_2$. In view of (\ref{2.4}) and (\ref{2.6}), one may confirm that the implicit 
hypothesis $s\ge s_\iota (\nu)$ ensures that $s>2(k-1)$. Consequently, there is an integer 
$\zet\equiv \eta \mmod{p^b}$ having the property that three at least of the variables $v_1,\ldots ,v_s$ 
are congruent to $\zet$ modulo $p^{b+1}$. Hence, again recalling the definitions (\ref{2.13}) and 
(\ref{2.14}), one finds by orthogonality in combination with H\"older's inequality that
\begin{align}
T_2&\le \binom{s}{3}\sum_{\substack{1\le \zet\le p^{b+1}\\ \zet\equiv \eta\mmod{p^b}}}
\oint |\grF_a^{k-1}(\bfalp;\xi)|^2\grf_{b+1}(\bfalp;\zet)^3\grf_b(\bfalp;\eta)^{s-3}
\grf_b(-\bfalp;\eta)^s\d\bfalp \notag \\
&\ll M\max_{\substack{1\le \zet \le p^{b+1}\\ \zet\equiv \eta \mmod{p^b}}}
(I_{a,b}^{k-1}(X;\xi,\eta))^{1-3/(2s)}(I_{a,b+1}^{k-1}(X;\xi,\zet))^{3/(2s)}.\label{4.3}
\end{align}

\par By substituting (\ref{4.2}) and (\ref{4.3}) into (\ref{4.1}), and recalling (\ref{2.19}) and
 (\ref{2.20}), we therefore conclude that
$$I_{a,b}^{k-1}(X)\ll (K_{a,b}^{k-1}(X))^{1/2}(I_{a,b}^{k-1}(X))^{1/2}+
M(I_{a,b}^{k-1}(X))^{1-3/(2s)}(I_{a,b+1}^{k-1}(X))^{3/(2s)},$$
whence
$$I_{a,b}^{k-1}(X)\ll K_{a,b}^{k-1}(X)+M^{2s/3}I_{a,b+1}^{k-1}(X).$$
This completes the proof of the lemma.
\end{proof}

Repeated application of Lemma \ref{lemma4.1}, together with a trivial bound for the mean value 
$K_{a,b+H}^{k-1}(X)$ when $H$ is large enough, yields a relation suitable for our iterative process.

\begin{lemma}\label{lemma4.2} Let $a$ and $b$ be integers with $1\le a<b$, and put $H=15(b-a)$. 
Suppose that $b+H\le (2\tet)^{-1}$. Then there exists an integer $h$ with $0\le h<H$ having the 
property that
$$I_{a,b}^{k-1}(X)\ll (M^h)^{2s/3}K_{a,b+h}^{k-1}(X)+(M^H)^{-s/4}(X/M^b)^{2s}
(X/M^a)^{\lam-2s}.$$
\end{lemma}

\begin{proof} By repeated application of Lemma \ref{lemma4.1}, we obtain the upper bound
\begin{equation}\label{4.4}
I_{a,b}^{k-1}(X)\ll \sum_{h=0}^{H-1}(M^h)^{2s/3}K_{a,b+h}^{k-1}(X)+
(M^H)^{2s/3}I_{a,b+H}^{k-1}(X).
\end{equation}
On considering the underlying Diophantine systems, it follows from Lemma \ref{lemma2.2} that, 
uniformly in $\xi$ and $\eta$, one has
\begin{align*}
I_{a,b+H}^{k-1}(X;\xi,\eta)&\le \oint|\grf_a(\bfalp;\xi)^{2k-2}\grf_{b+H}(\bfalp;\eta)^{2s}|\d\bfalp \\
&\ll \left(J_\grw(X/M^a)\right)^{(k-1)/\grw}\left(J_\grw(X/M^{b+H})\right)^{1-(k-1)/\grw}.
\end{align*}
The argument completing the proof of \cite[Lemma 4.2]{FW2013} now applies, delivering the estimate
$$(M^H)^{2s/3}I_{a,b+H}^{k-1}(X)\ll (M^H)^{-s/4}(X/M^b)^{2s}(X/M^a)^{\lam-2s},$$
and the conclusion of the lemma follows on substituting this bound into (\ref{4.4}).
\end{proof}

\section{The precongruencing step} It is necessary to configure the variables in the initial step of our 
iteration so that subsequent iterations are not impeded. Here we are able to make use of our earlier work 
\cite[\S6]{Woo2013} and \cite[\S6]{FW2013} concerning precongruencing steps so as to abbreviate the 
discussion, despite our present alteration of the definition of $K_{a,b}^m(X)$ relative to its earlier 
analogues.

\begin{lemma}\label{lemma5.1}
There exists a prime number $p$ with $M<p\le 2M$, and an integer $h$ with $h\in \{ 0,1,2,3\}$, for
 which one has
$$J_\grw (X)\ll M^{2s+2sh/3}K_{0,1+h}^{k-1}(X).$$
\end{lemma}

\begin{proof} The argument of the proof of \cite[Lemma 6.1]{FW2013}, leading via equation (6.2) to 
equation (6.3) of that paper, shows that there is a prime number $p$ with $M<p\le 2M$ for which
$$J_\grw(X)\ll p^{2s}\max_{1\le \eta\le p}\Itil_1^{k-1}(X;\eta).$$
By modifying the argument of the proof of \cite[Lemma 6.1]{FW2013} leading to equation (6.6) of that 
paper, along the lines easily surmised from our proof of Lemma \ref{lemma4.1} above, one finds that
\begin{equation}\label{5.1}
\Itil_c^{k-1}(X;\eta)\ll \Ktil_c^{k-1}(X;\eta)+M^{2s/3}\max_{1\le \zet\le p^{c+1}}
\Itil_{c+1}^{k-1}(X;\zet).
\end{equation}

\par Next, we iterate (\ref{5.1}) in order to bound $\Itil_1^{k-1}(X;\eta)$, just as in the argument 
concluding the proof of \cite[Lemma 6.1]{FW2013}. In this way, we find either that
\begin{equation}\label{5.2}
J_\grw(X)\ll M^{2s+2sh/3}\max_{1\le \zet\le p^{1+h}}\Ktil_{1+h}^{k-1}(X;\zet)
\end{equation}
for some index $h\in \{0,1,2,3\}$, so that the conclusion of the lemma holds by virtue of the definition 
(\ref{2.22}), or else that
$$J_\grw(X)\ll X^{\lam+\del}M^{-\grw/3}.$$
Thus, on recalling the definition (\ref{2.9}) of $\del$, we find that $J_\grw(X)\ll X^{\lam-2\del}$, 
contradicting the lower bound (\ref{2.10}) whenever $X=X_l$ is sufficiently large. We are therefore forced 
to conclude that the earlier upper bound (\ref{5.2}) holds, and hence the proof of the lemma is complete.
\end{proof}

It is at this point that we fix the prime number $p$, once and for all, in accordance with Lemma
 \ref{lemma5.1}.

\section{The efficient congruencing step} We extract congruence information from the mean value 
$K_{a,b}^{k-1}(X)$ in two phases. The reader familiar with earlier efficient congruencing arguments will 
identify significant complications in each phase associated with our (forced) inhomogeneous definition 
(\ref{2.15}) of the mean value $K_{a,b}^{k-1}(X;\xi,\eta)$. Before describing the first phase of the 
efficient congruencing step, in which we relate $K_{a,b}^{k-1}(X)$ to $I_{b,(k-r)b}^{k-1}(X)$ and 
$K_{a,b}^r(X)$, we introduce some additional notation. We define the generating function
\begin{equation}\label{6.1}
\grH_{c,d}^m(\bfalp;\xi)=\sum_{\bfxi\in \Xi_c^m(\xi)}
\sum_{\substack{1\le \bfzet\le p^d\\ \bfzet\equiv \bfxi\mmod{p^{c+1}}}}\prod_{i=1}^m
|\grf_d(\bfalp;\zet_i)|^2,
\end{equation}
adopting the natural convention that $\grH_{c,d}^0(\bfalp;\xi)=1$. For future reference, we note at this
 point that successive applications of H\"older's inequality show that when $\ome$ is a real number with
 $m\ome\ge 1$, then
\begin{align}
\grH_{c,d}^m(\bfalp;\xi)^\ome &\le 
\Biggl( \sum_{\substack{1\le \zet \le p^d\\ \zet \equiv \xi\mmod{p^c}}}
|\grf_d(\bfalp;\zet)|^2\Biggr)^{m\ome } \notag \\
&\le (p^{d-c})^{m\ome -1}\sum_{\substack{1\le \zet \le p^d\\ \zet \equiv \xi\mmod{p^c}}}
|\grf_d(\bfalp;\zet)|^{2m\ome}.\label{6.2}
\end{align}
Finally, we recall the definitions of $\mu$ and $\nu$ from (\ref{3.4}). 

\begin{lemma}\label{lemma6.1} Suppose that $a$ and $b$ are integers with $0\le a<b\le \tet^{-1}$ and
 $b\ge (r-1)a$. Then one has
$$K_{a,b}^{k-1}(X)\ll M^{\mu b+\nu a}
\left( (M^{(k-r)b-a})^sI_{b,(k-r)b}^{k-1}(X)\right)^{\tfrac{k-r-1}{s-r}}
\left( K_{a,b}^r(X)\right)^{\tfrac{s-k+1}{s-r}}.$$
\end{lemma}

\begin{proof} We first consider the situation in which $a\ge 1$. The argument associated with the case 
$a=0$ is very similar, and so we are able to appeal later to a highly abbreviated argument for this case to
 complete the proof of the lemma. Consider fixed integers $\xi$ and $\eta$ with
\begin{equation}\label{6.3}
1\le \xi\le p^a,\quad 1\le \eta\le p^b\quad \text{and}\quad \eta\not\equiv \xi\mmod{p}.
\end{equation}
The quantity $K_{a,b}^{k-1}(X;\xi,\eta)$ counts the number of integral solutions of the system 
(\ref{2.17}) with $m=k-1$ subject to the attendant conditions on $\bfx$, $\bfy$, $\bfu$, $\bfv$, $\bfw$, 
$\bfz$. Given such a solution of the system (\ref{2.17}), the discussion leading to (\ref{2.18}) shows that
\begin{equation}\label{6.4}
\sum_{i=1}^{k-1}(x_i-\eta)^j\equiv \sum_{i=1}^{k-1}(y_i-\eta)^j\mmod{p^{jb}}\quad (1\le j\le k).
\end{equation}
In the notation introduced in \S3, it follows that for some $k$-tuple of integers $\bfm$, both 
$[\bfx\mmod{p^{kb}}]$ and $[\bfy\mmod{p^{kb}}]$ lie in $\calB_{a,b}^{k-1}(\bfm;\xi,\eta)$. Write
$$\grG_{a,b}(\bfalp;\bfm)=\sum_{\bftet\in \calB_{a,b}^{k-1}(\bfm;\xi,\eta)} 
\prod_{i=1}^{k-1}\grf_{kb}(\bfalp;\tet_i).$$
Then on considering the underlying Diophantine system, we see from (\ref{2.17}) and (\ref{6.4}) that 
$$K_{a,b}^{k-1}(X;\xi,\eta)=\sum_{m_1=1}^{p^b}\ldots \sum_{m_k=1}^{p^{kb}}
\oint |\grG_{a,b}(\bfalp;\bfm)^2\grF^*(\bfalp)^2|\d\bfalp ,$$
where
\begin{equation}\label{6.5}
\grF^*(\bfalp)=\grF_b^{k-1}(\bfalp;\eta)\grf_b(\bfalp;\eta)^{s-k+1}.
\end{equation}

\par We now partition the vectors in each set $\calB_{a,b}^{k-1}(\bfm;\xi,\eta)$ into equivalence classes 
modulo $p^{(k-r)b}$ as in \S3. Write $\calC(\bfm)=\calC_{a,b}^{k-1,k-r}(\bfm;\xi,\eta)$. By applying 
Cauchy's inequality and then recalling (\ref{3.2}), we find by means of Lemma \ref{lemma3.1} that
\begin{align*}
|\grG_{a,b}(\bfalp;\bfm)|^2&=\Bigl| \sum_{\grC \in \calC(\bfm)}
\sum_{\bftet \in \grC}\prod_{i=1}^{k-1}\grf_{kb}(\bfalp;\tet_i)\Bigr|^2\\
&\le \text{card}(\calC(\bfm))\sum_{\grC\in \calC(\bfm)}
\Bigl| \sum_{\bftet \in \grC}\prod_{i=1}^{k-1}\grf_{kb}(\bfalp;\tet_i)\Bigr|^2\\
&\ll M^{\mu b+\nu a}\sum_{\grC\in \calC(\bfm)}
\Bigl| \sum_{\bftet \in \grC}\prod_{i=1}^{k-1}\grf_{kb}(\bfalp;\tet_i)\Bigr|^2.
\end{align*}
Hence
$$K_{a,b}^{k-1}(X;\xi,\eta)\ll M^{\mu b+\nu a}\sum_\bfm \sum_{\grC\in \calC(\bfm)}\oint 
\Bigl| \grF^*(\alp)\sum_{\bftet \in \grC}\prod_{i=1}^{k-1}\grf_{kb}(\bfalp;\tet_i)\Bigr|^2\d \bfalp.$$
For each $k$-tuple $\bfm$ and equivalence class $\grC$, the integral above counts solutions of 
(\ref{2.17}) with the additional constraint that both $[\bfx \mmod{p^{kb}}]$ and 
$[\bfy \mmod{p^{kb}}]$ lie in $\grC$. In particular, one has $\bfx\equiv \bfy\mmod{p^{(k-r)b}}$. 
Moreover, as the sets $\calB_{a,b}^{k-1}(\bfm;\xi,\eta)$ are disjoint for distinct $k$-tuples $\bfm$ with 
$1\le m_j\le p^{jb}$ $(1\le j\le k)$, to each pair $(\bfx,\bfy)$ there corresponds at most one pair 
$(\bfm,\grC)$. Thus we deduce that
$$K_{a,b}^{k-1}(X;\xi,\eta)\ll M^{\mu b+\nu a}H,$$
where $H$ denotes the number of solutions of (\ref{2.17}) subject to the additional condition 
$\bfx\equiv \bfy\mmod{p^{(k-r)b}}$. Hence, on considering the underlying Diophantine systems and 
recalling (\ref{6.1}), we discern that
\begin{equation}\label{6.6}
K_{a,b}^{k-1}(X;\xi,\eta)\ll M^{\mu b+\nu a}\oint \grH_{a,(k-r)b}^{k-1}(\bfalp;\xi)
|\grF^*(\bfalp)|^2\d\bfalp .
\end{equation}

\par An inspection of the definition of $\Xi_a^m(\xi)$ in the preamble to (\ref{2.13}) reveals that when
 $\bfxi\in \Xi_a^{k-1}(\xi)$, then in particular one has
\begin{equation}\label{6.7}
(\xi_1,\ldots ,\xi_r)\in \Xi_a^r(\xi)\quad \text{and}\quad (\xi_{r+1},\ldots ,\xi_{k-1})\in 
\Xi_a^{k-r-1}(\xi).\end{equation}
Note here that in the situation with $r=k-1$, the second of these conditions is interpreted as vacuous. In 
view of (\ref{6.7}), a further consideration of the underlying Diophantine systems leads from (\ref{6.6})
 via (\ref{6.1}) to the upper bound
$$K_{a,b}^{k-1}(X;\xi,\eta)\ll M^{\mu b+\nu a}\oint \grH_{a,(k-r)b}^r(\bfalp;\xi)
\grH_{a,(k-r)b}^{k-r-1}(\bfalp;\xi)|\grF^*(\bfalp)|^2\d\bfalp .$$
By applying H\"older's inequality to the integral on the right hand side of this relation, keeping in mind the 
definition (\ref{6.5}), we obtain the bound
\begin{equation}\label{6.8}
K_{a,b}^{k-1}(X;\xi,\eta)\ll M^{\mu b+\nu a}U_1^{\ome_1}U_2^{\ome_2}U_3^{\ome_3},
\end{equation}
where
\begin{equation}\label{6.9}
\ome_1=\frac{s-k+1}{s-r},\quad \ome_2=\frac{k-r-1}{s},\quad \ome_3=\frac{r(k-r-1)}{s(s-r)},
\end{equation}
and
\begin{equation}\label{6.10}
U_1=\oint \grH_{a,(k-r)b}^r(\bfalp;\xi)|\grF_b^{k-1}(\bfalp;\eta)^2\grf_b(\bfalp;\eta)^{2s-2r}|
\d\bfalp ,
\end{equation}
\begin{equation}\label{6.11}
U_2=\oint |\grF_b^{k-1}(\bfalp;\eta)|^2\grH_{a,(k-r)b}^{k-r-1}(\bfalp;\xi)^{s/(k-r-1)}\d\bfalp ,
\end{equation}
\begin{equation}\label{6.12}
U_3=\oint |\grF_b^{k-1}(\bfalp;\eta)|^2\grH_{a,(k-r)b}^r(\bfalp;\xi)^{s/r}\d\bfalp .
\end{equation}
We note here that the condition (\ref{2.4}) ensures that $r\le k-2$, so that $\ome_1$, $\ome_2$ and 
$\ome_3$ are each positive. Thus the argument leading to (\ref{6.8}) represents a legitimate application 
of H\"older's inequality.\par

Our next task is to relate the mean values $U_i$ to the more familiar ones introduced in \S2. Observe first
 that a consideration of the underlying Diophantine system leads from (\ref{6.10}) via (\ref{6.1}) and
 (\ref{2.13}) to the upper bound
\begin{equation}\label{6.13}
U_1\le \oint |\grF_a^r(\bfalp;\xi)^2\grF_b^{k-1}(\bfalp;\eta)^2\grf_b(\bfalp;\eta)^{2s-2r}|
\d\bfalp.
\end{equation}
Indeed, the Diophantine system underlying the mean value $U_1$ is subject to additional diagonal
 structure that we have discarded in the mean value on the right hand side of (\ref{6.13}). It is worth
 noting here that this man\oe uvre, though superficially inefficient, loses nothing in the ensuing argument,
 since the additional diagonal constraint is recovered without cost in the next stage of our argument, in
 Lemma \ref{lemma6.2}. On recalling (\ref{2.15}) and (\ref{2.20}), we thus deduce from (\ref{6.13})
 that
\begin{equation}\label{6.14}
U_1\le K_{a,b}^r(X).
\end{equation}
Next, by employing (\ref{6.2}) within (\ref{6.11}) and (\ref{6.12}), we find that
$$U_2+U_3\ll (M^{(k-r)b-a})^s\max_{\substack{1\le \zet\le p^{(k-r)b}\\ \zet\equiv \xi\mmod{p^a}}}
\oint |\grF_b^{k-1}(\bfalp ;\eta)^2\grf_{(k-r)b}(\bfalp;\zet)^{2s}|\d\bfalp .$$
Notice here that since the condition (\ref{6.3}) implies that $\eta\not\equiv \xi\mmod{p}$, and we have 
$\zet\equiv \xi\mmod{p^a}$ with $a\ge 1$, then $\zet\not\equiv \eta\mmod{p}$. In this way we deduce
 from (\ref{2.14}) and (\ref{2.19}) that
\begin{equation}\label{6.15}
U_2+U_3\ll (M^{(k-r)b-a})^s I_{b,(k-r)b}^{k-1}(X).
\end{equation}

By substituting (\ref{6.14}) and (\ref{6.15}) into the relation
$$K_{a,b}^{k-1}(X;\xi,\eta )\ll M^{\mu b+\nu a}U_1^{\ome_1}(U_2+U_3)^{1-\ome_1},$$
that is immediate from (\ref{6.8}), and then recalling (\ref{6.9}) and (\ref{2.20}), the conclusion of the
 lemma follows when $a\ge 1$. When $a=0$, we must modify this argument slightly. In this case, from 
(\ref{2.21}) and (\ref{2.22}) we find that
$$K_{0,b}^{k-1}(X)=\max_{1\le \eta\le p^b}\oint |\grF^{k-1}(\bfalp;\eta)^2
\grF_b^{k-1}(\bfalp;\eta)^2\grf_b(\bfalp;\eta)^{2s-2k+2}|\d \bfalp.$$
The desired conclusion follows in this instance by pursuing the proof given above in the case $a\ge 1$,
 noting that the definition of $\grF^{k-1}(\bfalp;\eta)$ ensures that the variables resulting from the 
congruencing argument avoid the congruence class $\eta$ modulo $p$. This completes the proof of the 
lemma.
\end{proof}

We now establish the machinery for an iteration by relating the mean value $K_{a,b}^{r-j+1}(X)$ to 
$I_{b,(k-r+j)b}^{k-1}(X)$ and $K_{a,b}^{r-j}(X)$ for $j=1,\ldots ,r$. Each step of this iteration 
effectively extracts (approximately) two variables mutually congruent modulo $p^{(k-r+j)b}$ for use in
 the next stage of the efficient congruencing argument, leaving a mean value of similar type to the original 
one on which a stronger congruencing process is applicable. It is at this point that we make effective the 
argument sketched at the end of the introduction. It is useful here and later to write
\begin{equation}\label{6.16}
k_j=k-r+j\quad \text{and}\quad s_j=s-r+j.
\end{equation}

\begin{lemma}\label{lemma6.2} Suppose that $a$ and $b$ are integers with $0\le a<b\le \tet^{-1}$ and 
$b\ge (r-1)a$. Then for $1\le j\le r$, one has
$$K_{a,b}^{r-j+1}(X)\ll M^{(r-j)a}
\left( (M^{k_jb-a})^sI_{b,k_jb}^{k-1}(X)\right)^{1/s_j}
\left( K_{a,b}^{r-j}(X)\right)^{s_{j-1}/s_j}.$$
\end{lemma}

\begin{proof} We follow closely the argument of the proof of previous lemma. We again suppose in the 
first instance that $a\ge 1$. Consider fixed integers $\xi$ and $\eta$ satisfying the conditions imposed by 
(\ref{6.3}). The quantity $K_{a,b}^{r-j+1}(X;\xi,\eta)$ counts the number of integral solutions of the
 system (\ref{2.17}) with $m=r-j+1$ subject to the attendant conditions on $\bfx$, $\bfy$, $\bfu$,
 $\bfv$, $\bfw$, $\bfz$. Given such a solution of the system (\ref{2.17}), the argument leading to
 (\ref{2.18}) shows that
\begin{equation}\label{6.17}
\sum_{i=1}^{r-j+1}(x_i-\eta)^l\equiv \sum_{i=1}^{r-j+1}(y_i-\eta)^l\mmod{p^{lb}}\quad 
(1\le l\le k).
\end{equation}
In the notation introduced in \S3, it follows that for some $k$-tuple of integers $\bfm$, both 
$[\bfx\mmod{p^{kb}}]$ and $[\bfy\mmod{p^{kb}}]$ lie in $\calB_{a,b}^{r-j+1}(\bfm;\xi,\eta)$. Write
$$\grG_{a,b}(\bfalp;\bfm)=\sum_{\bftet\in \calB_{a,b}^{r-j+1}(\bfm;\xi,\eta)} 
\prod_{i=1}^{r-j+1}\grf_{kb}(\bfalp;\tet_i).$$
Then on considering the underlying Diophantine system, we see from (\ref{2.17}) and (\ref{6.17}) that 
$$K_{a,b}^{r-j+1}(X;\xi,\eta)=\sum_{m_1=1}^{p^b}\ldots \sum_{m_k=1}^{p^{kb}}
\oint |\grG_{a,b}(\bfalp;\bfm)^2\grF_j^*(\bfalp)^2|\d\bfalp ,$$
where we write
\begin{equation}\label{6.18}
\grF_j^*(\bfalp)=\grF_b^{k-1}(\bfalp;\eta)\grf_b(\bfalp;\eta)^{s_{j-1}}.
\end{equation}

\par We now partition the vectors in each set $\calB_{a,b}^{r-j+1}(\bfm;\xi,\eta)$ into equivalence 
classes modulo $p^{k_jb}$ as in \S3. Write $\calC(\bfm)=\calC_{a,b}^{r-j+1,k_j}(\bfm;\xi,\eta)$. By 
applying Cauchy's inequality and then recalling (\ref{3.2}), we find by means of Lemma \ref{lemma3.1} 
that
\begin{align*}
|\grG_{a,b}(\bfalp;\bfm)|^2&=\Bigl| \sum_{\grC \in \calC(\bfm)}
\sum_{\bftet \in \grC}\prod_{i=1}^{r-j+1}\grf_{kb}(\bfalp;\tet_i)\Bigr|^2\\
&\le \text{card}(\calC(\bfm))\sum_{\grC\in \calC(\bfm)}
\Bigl| \sum_{\bftet \in \grC}\prod_{i=1}^{r-j+1}\grf_{kb}(\bfalp;\tet_i)\Bigr|^2\\
&\ll M^{(r-j)a}\sum_{\grC\in \calC(\bfm)}
\Bigl| \sum_{\bftet \in \grC}\prod_{i=1}^{r-j+1}\grf_{kb}(\bfalp;\tet_i)\Bigr|^2.
\end{align*}
Hence
$$K_{a,b}^{r-j+1}(X;\xi,\eta)\ll M^{(r-j)a}\sum_\bfm \sum_{\grC\in \calC(\bfm)}\oint 
\Bigl| \grF_j^*(\alp)\sum_{\bftet \in \grC}\prod_{i=1}^{r-j+1}\grf_{kb}(\bfalp;\tet_i)\Bigr|^2
\d \bfalp.$$
For each $k$-tuple $\bfm$ and equivalence class $\grC$, the integral above counts solutions of 
(\ref{2.17}) with the additional constraint that both $[\bfx \mmod{p^{kb}}]$ and 
$[\bfy \mmod{p^{kb}}]$ lie in $\grC$. In particular, one has $\bfx\equiv \bfy\mmod{p^{k_jb}}$. 
Moreover, as the sets $\calB_{a,b}^{r-j+1}(\bfm;\xi,\eta)$ are disjoint for distinct $k$-tuples $\bfm$ with 
$1\le m_j\le p^{jb}$ $(1\le j\le k)$, to each pair $(\bfx,\bfy)$ there corresponds at most one pair 
$(\bfm,\grC)$. Thus we deduce that
$$K_{a,b}^{r-j+1}(X;\xi,\eta)\ll M^{(r-j)a}H,$$
where $H$ denotes the number of solutions of (\ref{2.17}) subject to the additional condition 
$\bfx\equiv \bfy\mmod{p^{k_jb}}$. Hence, on considering the underlying Diophantine systems and 
recalling (\ref{6.1}), we discern that
\begin{equation}\label{6.19}
K_{a,b}^{r-j+1}(X;\xi,\eta)\ll M^{(r-j)a}\oint \grH_{a,k_jb}^{r-j+1}(\bfalp;\xi)
|\grF_j^*(\bfalp)|^2\d\bfalp .
\end{equation}

\par An inspection of the definition of $\Xi_a^m(\xi)$, given in the preamble to (\ref{2.13}), reveals on
 this occasion that when $\bfxi\in \Xi_a^{r-j+1}(\xi)$, then
$$(\xi_1,\ldots ,\xi_{r-j})\in \Xi_a^{r-j}(\xi)\quad \text{and}\quad (\xi_{r-j+1})\in 
\Xi_a^{1}(\xi).$$
Then a further consideration of the underlying Diophantine systems leads from (\ref{6.19}) via (\ref{6.1}) 
to the upper bound
$$K_{a,b}^{r-j+1}(X;\xi,\eta)\ll M^{(r-j)a}\oint \grH_{a,k_jb}^{r-j}(\bfalp;\xi)\grH_{a,k_jb}^{1}
(\bfalp;\xi)|\grF_j^*(\bfalp)|^2\d\bfalp .$$
By applying H\"older's inequality to the integral on the right hand side of this relation, bearing in mind the 
definitions (\ref{6.16}) and (\ref{6.18}), we obtain the bound
\begin{equation}\label{6.20}
K_{a,b}^{r-j+1}(X;\xi,\eta)\ll M^{(r-j)a}U_1^{\ome_1}U_2^{\ome_2}U_3^{\ome_3},
\end{equation}
where
\begin{equation}\label{6.21}
\ome_1=s_{j-1}/s_j,\quad \ome_2=1/s,\quad \ome_3=(r-j)/(ss_j),
\end{equation}
and
\begin{equation}\label{6.22}
U_1=\oint \grH_{a,k_jb}^{r-j}(\bfalp;\xi)|\grF_b^{k-1}(\bfalp;\eta)^2\grf_b
(\bfalp;\eta)^{2s_j}|\d\bfalp ,
\end{equation}
\begin{equation}\label{6.23}
U_2=\oint |\grF_b^{k-1}(\bfalp;\eta)|^2\grH_{a,k_jb}^{1}(\bfalp;\xi)^s\d\bfalp ,
\end{equation}
\begin{equation}\label{6.24}
U_3=\oint |\grF_b^{k-1}(\bfalp;\eta)|^2\grH_{a,k_jb}^{r-j}(\bfalp;\xi)^{s/(r-j)}\d\bfalp .
\end{equation}

\par We again relate the mean values $U_i$ to those introduced in \S2. Observe first that a consideration
 of the underlying Diophantine system leads from (\ref{6.22}) via (\ref{6.1}) and (\ref{2.13}) to the upper
 bound
$$U_1\le \oint |\grF_a^{r-j}(\bfalp;\xi)^2\grF_b^{k-1}(\bfalp;\eta)^2\grf_b(\bfalp;\eta)^{2s_j}|
\d\bfalp.$$
On recalling (\ref{2.15}) and (\ref{2.20}), we thus deduce that
\begin{equation}\label{6.25}
U_1\le K_{a,b}^{r-j}(X).
\end{equation}
Next, by employing (\ref{6.2}) within (\ref{6.23}) and (\ref{6.24}), we find that
$$U_2+U_3\ll (M^{k_jb-a})^s\max_{\substack{1\le \zet\le p^{k_jb}\\ \zet\equiv \xi\mmod{p^a}}}
\oint |\grF_b^{k-1}(\bfalp ;\eta)^2\grf_{k_jb}(\bfalp;\zet)^{2s}|\d\bfalp .$$
Notice here that since the condition (\ref{6.3}) implies that $\eta\not\equiv \xi\mmod{p}$, and we have 
$\zet\equiv \xi\mmod{p^a}$ with $a\ge 1$, then once more one has $\zet\not\equiv \eta\mmod{p}$. In
 this way we deduce from (\ref{2.14}) and (\ref{2.19}) that
\begin{equation}\label{6.26}
U_2+U_3\ll (M^{k_jb-a})^s I_{b,k_jb}^{k-1}(X).
\end{equation}

By substituting (\ref{6.25}) and (\ref{6.26}) into the relation
$$K_{a,b}^{r-j+1}(X;\xi,\eta)\ll M^{(r-j)a}U_1^{\ome_1}(U_2+U_3)^{1-\ome_1},$$
that is immediate from (\ref{6.20}), and then recalling (\ref{6.21}) and (\ref{2.20}), the conclusion of the
 lemma follows when $a\ge 1$. When $a=0$, we must modify this argument slightly. In this case, from 
(\ref{2.21}) and (\ref{2.22}) we find that
$$K_{0,b}^{r-j+1}(X)=\max_{1\le \eta\le p^b}\oint |\grF^{r-j+1}(\bfalp;\eta)^2
\grF_b^{k-1}(\bfalp;\eta)^2\grf_b(\bfalp;\eta)^{2s_{j-1}}|\d \bfalp.$$
The desired conclusion follows in this instance by pursuing the proof given above in the case $a\ge 1$,
 noting that the definition of $\grF^{r-j+1}(\bfalp;\eta)$ ensures that the variables resulting from the 
congruencing argument avoid the congruence class $\eta$ modulo $p$. This completes the proof of the 
lemma.
\end{proof}

There are, of course, similarities between the arguments applied to establish Lemmata \ref{lemma6.1} and 
\ref{lemma6.2}. Some economy of space would be afforded by the proof of a common lemma, of which 
these respective lemmata would be special cases. However, the considerable complications associated with 
such a unified approach would, on the one hand, obscure the strategy underlying the proof of these 
lemmata, and on the other hand consume not inconsiderable space to accommodate these complications. 
Thus, we have deliberately opted for clarity over concision in offering two separate treatments.

\section{The multigrade combination} We next combine the estimates supplied by Lemmata 
\ref{lemma6.1} and \ref{lemma6.2} so as to bound $K_{a,b}^{k-1}(X)$ in terms of the mean values
 $I_{b,k_jb}^{k-1}(X)$ $(0\le j\le r)$. We achieve this goal by initiating this process with Lemma
 \ref{lemma6.1}, and then iterate the application of Lemma \ref{lemma6.2}. Before announcing our basic
 asymptotic estimate, we recall the definition (\ref{6.16}) and then define the exponents
\begin{equation}\label{7.1}
\phi_j=(s-k+1)/(s_{j-1}s_j)\quad (1\le j\le r).
\end{equation}
In addition, we write
\begin{equation}\label{7.2}
\phi_0=(k-r-1)/s_0\quad \text{and}\quad \phi^*=(s-k+1)/s,
\end{equation}
so that
\begin{align}
\phi^*+\sum_{j=0}^r\phi_j=&\, \frac{s-k+1}{s}+\frac{k-r-1}{s-r}+
(s-k+1)\sum_{j=1}^r(s_{j-1}^{-1}-s_j^{-1})\notag \\
=&\,\frac{k-r-1}{s-r}+\frac{s-k+1}{s-r}=1.\label{7.3}
\end{align}
Notice here that $\phi_j$ is roughly equal to $1/s$ for $1\le j\le r$. With this in mind, the reader will find
 that the conclusion of our next lemma is an approximate analogue of the formula presented in the display
 preceding \cite[equation (11.3)]{Woo2012a}, a key element in the heuristic argument that inspired our 
present work.

\begin{lemma}\label{lemma7.1}
Suppose that $a$ and $b$ are integers with $0\le a<b\le \tet^{-1}$ and $b\ge (r-1)a$. Then one has
$$K_{a,b}^{k-1}(X)\ll M^{\mu'b+\nu'a}\left(J_\grw(X/M^b)\right)^{\phi^*}
\prod_{j=0}^r\left( I_{b,k_jb}^{k-1}(X)\right)^{\phi_j},$$
where
\begin{equation}\label{7.4}
\mu'=\mu+sk_0(k_0-1)/s_0+s(s-k+1)\sum_{j=1}^rk_j/(s_{j-1}s_j)
\end{equation}
and
\begin{equation}\label{7.5}
\nu'=\nu-s(k_0-1)/s_0+(s-k+1)\sum_{j=1}^rs_{j-1}^{-1}\left( r-j-s/s_j\right) .
\end{equation}
\end{lemma}

\begin{proof} We prove by induction that for $0\le l\le r$, one has
\begin{equation}\label{7.6}
K_{a,b}^{k-1}(X)\ll M^{\mu_lb+\nu_la}\left( K_{a,b}^{r-l}(X)\right)^{\phi^*_l}
\prod_{j=0}^l\left( I_{b,k_jb}^{k-1}(X)\right)^{\phi_j},
\end{equation}
where
$$\phi^*_l=(s-k+1)/s_l,$$
$$\mu_l=\mu+sk_0(k_0-1)/s_0+s(s-k+1)\sum_{j=1}^lk_j/(s_{j-1}s_j)$$
and
$$\nu_l=\nu-s(k_0-1)/s_0+(s-k+1)\sum_{j=1}^ls_{j-1}^{-1}(r-j-s/s_j).$$
The conclusion of the lemma follows from the case $l=r$ of (\ref{7.6}), on noting that Lemma 
\ref{lemma2.1} delivers the estimate $K_{a,b}^0(X)\ll J_\grw(X/M^b)$.\par

We observe first that the inductive hypothesis (\ref{7.6}) holds when $l=0$, as a consequence of Lemma
 \ref{lemma6.1}, definition (\ref{6.16}), and the familiar convention that an empty sum is zero. Suppose
 then that $J$ is a positive integer not exceeding $r$, and that the inductive hypothesis (\ref{7.6}) holds
 for $0\le l<J$. An application of Lemma \ref{lemma6.2} supplies the estimate
$$K_{a,b}^{r-J+1}(X)\ll M^{(r-J)a}
\left( (M^{k_Jb-a})^sI_{b,k_Jb}^{k-1}(X)\right)^{1/s_J}\left( K_{a,b}^{r-J}(X)\right)^{s_{J-1}/s_J}.$$
On substituting this bound into the estimate (\ref{7.6}) with $l=J-1$, one obtains the new upper bound
$$K_{a,b}^{k-1}(X)\ll M^\Ome \left( K_{a,b}^{r-J}(X)\right)^{\phi^*_J}
\prod_{j=0}^J\left( I_{b,k_jb}^{k-1}(X)\right)^{\phi_j},$$
where
$$\Ome=\mu_{J-1}b+\nu_{J-1}a+\left( (r-J)a+s(k_Jb-a)/s_J\right) \phi^*_{J-1}.$$
Since
$$\mu_J=\mu_{J-1}+s(s-k+1)k_J/(s_{J-1}s_J),$$
and
$$\nu_J=\nu_{J-1}+(s-k+1)s_{J-1}^{-1}\left( r-J-s/s_J\right),$$
we find that the estimate (\ref{7.6}) holds with $l=J$, completing the proof of the inductive step. In view 
of our earlier remarks, the conclusion of the lemma now follows.
\end{proof}

We next recall the anticipated magnitude operator $\llbracket \,\cdot \,\rrbracketsub{{\grd,\rho}}$ defined 
in equations (\ref{2.23}) to (\ref{2.25}), and convert Lemma \ref{lemma7.1} into a more portable form. 
Before announcing our conclusions, we recall the definition (\ref{2.27}) of $\Lam(\Del)$.

\begin{lemma}\label{lemma7.2} Suppose that $a$ and $b$ are integers with $0\le a<b\le \tet^{-1}$ and 
$b\ge (r-1)a$. Then one has
$$\llbracket K_{a,b}^{k-1}(X)\rrbracketsub{{\Del^\dag,1}}\ll 
\left( (X/M^b)^{\Lam(\Del^\dag)+\del}\right)^{\phi^*}\prod_{j=0}^r
\llbracket I_{b,k_jb}^{k-1}(X)\rrbracketsub{{\Del^\dag,1}}^{\phi_j},$$
where
$$\Del^\dag=\sum_{m=1}^r\frac{(m-1)(k-m-1)}{s-m}.$$
\end{lemma}

\begin{proof} In order to promote greater transparency, we begin by utilising the operator 
$\llbracket {\,\cdot \,}\rrbracketsub{{\Del^\dag,0}}$, only later translating our statements into analogues
 for the operator $\llbracket \,\cdot \,\rrbracketsub{{\Del^\dag,1}}$. Write 
$\Lam^\dag=\Lam(\Del^\dag)$,
\begin{equation}\label{7.7}
\kap=2k-2-\tfrac{1}{2}k(k+1)+\Del^\dag,
\end{equation}
and define $\mu'$ and $\nu'$ as in (\ref{7.4}) and (\ref{7.5}). Then we find from Lemma \ref{lemma7.1} 
in combination with (\ref{2.11}) that
\begin{equation}\label{7.8}
\llbracket K_{a,b}^{k-1}(X)\rrbracketsub{{\Del^\dag,0}}\ll M^{\mu^*b+\nu^*a}
\left( (X/M^b)^{\Lam^\dag+\del}\right)^{\phi^*}
\prod_{j=0}^r\llbracket I_{b,k_jb}^{k-1}(X)\rrbracketsub{{\Del^\dag,0}}^{\phi_j},
\end{equation}
where
$$\mu^*=\mu'+2s-\phi^*(\kap+2s)-\sum_{j=0}^r\phi_j(\kap+2sk_j)\quad \text{and}\quad
 \nu^*=\nu'+\kap.$$

\par On recalling (\ref{7.1}) to (\ref{7.5}), we find that
$$\mu^*=\mu+2k-2-\kap-sk_0(k_0-1)s_0^{-1}-s(s-k+1)\sum_{j=1}^rk_j/(s_{j-1}s_j).$$
But in view of (\ref{6.16}), one has
\begin{align}
\sum_{j=1}^r\frac{(s-k+1)k_j}{s_{j-1}s_j}&=\sum_{j=1}^r
\left( \frac{k_{j-1}k_j}{s_j}-\frac{k_{j-2}k_{j-1}}{s_{j-1}}-\frac{k_{j-2}}{s_{j-1}}\right)\notag \\
&=\frac{k_{r-1}k_r}{s_r}-\frac{k_0(k_0-1)}{s_0}-\sum_{j=1}^r\frac{k_{j-2}}{s_{j-1}}.\label{7.9}
\end{align}
Thus, again recalling (\ref{6.16}) and making the change of variable $m=r-j+1$, we discern that
$$\mu^*=\mu+\tfrac{1}{2}k(k+1)-\Del^\dag-k(k-1)+\sum_{m=1}^r\frac{s(k-m-1)}{s-m}.$$
Consequently, by reference now to (\ref{3.4}) and the definition of $\Del^\dag$, we obtain
$$\mu^*=\tfrac{1}{2}(k-r-1)(k-r-2)-\tfrac{1}{2}k(k-3)+
\sum_{m=1}^r\left( k-m-1+\frac{k-m-1}{s-m}\right),$$
whence
\begin{equation}\label{7.10}
\mu^*=1+\sum_{m=1}^r\frac{k-m-1}{s-m}.
\end{equation}

\par Observe next that
$$\sum_{j=1}^r\frac{s-k+1}{s_{j-1}s_j}=\sum_{j=1}^r
\left( \frac{k_{j-1}}{s_j}-\frac{k_{j-2}}{s_{j-1}}\right)=\frac{k_{r-1}}{s_r}-\frac{k_0-1}{s_0}.$$
Thus, in view of (\ref{6.16}), it follows from (\ref{7.5}) that
$$\nu^*=\nu+\kap-(k-1)+(s-k+1)\sum_{j=1}^r\frac{r-j}{s-r+j-1}.$$
On making the change of variable $m=r-j+1$ once again, we therefore obtain
$$\nu^*=\nu+\kap-(k-1)+\sum_{m=1}^r\left( m-1-\frac{(m-1)(k-m-1)}{s-m}\right).$$
In this way, now invoking (\ref{3.4}), (\ref{7.7}) and the definition of $\Del^\dag$, we arrive at
 the relation
\begin{equation}\label{7.11}
\nu^*=\tfrac{1}{2}(k-r-1)(k+r-2)+k-1-\tfrac{1}{2}k(k+1)+\tfrac{1}{2}r(r-1)=-k.
\end{equation}

\par At this point we note from (\ref{2.24}) that
$$\llbracket I_{b,k_jb}^{k-1}(X)\rrbracketsub{{\Del^\dag,1}}=M^{(k-k_j)b}
\llbracket I_{b,k_jb}^{k-1}(X)\rrbracketsub{{\Del^\dag,0}}.$$
Furthermore, in view of (\ref{6.16}), (\ref{7.1}) and (\ref{7.2}), one has the relation
\begin{align*}
\sum_{j=1}^r(k-k_j)\phi_j&=(s-k+1)\sum_{j=1}^r\frac{r-j}{s_{j-1}s_j}\\
&=\sum_{j=1}^r\left( \frac{(r-j)k_{j-1}}{s_j}-\frac{(r-j+1)k_{j-2}}{s_{j-1}}+\frac{k_{j-2}}{s_{j-1}}
\right),
\end{align*}
whence
$$\sum_{j=0}^r(k-k_j)\phi_j=\frac{r(k_0-1)}{s_0}+\sum_{j=1}^r(k-k_j)\phi_j=
\sum_{m=1}^r\frac{k-m-1}{s-m}.$$
Thus, collecting together our formulae (\ref{7.10}) and (\ref{7.11}) for $\mu^*$ and $\nu^*$ within 
(\ref{7.8}), we obtain the upper bound
$$\llbracket K_{a,b}^{k-1}(X)\rrbracketsub{{\Del^\dag,0}}\ll M^{b-ka}
\left( (X/M^b)^{\Lam^\dag+\del}\right)^{\phi^*}
\prod_{j=0}^r\llbracket I_{b,k_jb}^{k-1}(X)\rrbracketsub{{\Del^\dag,1}}^{\phi_j}.$$
The conclusion of the lemma now follows from (\ref{2.25}), since the latter implies the relation
$$M^{ka-b}\llbracket K_{a,b}^{k-1}(X)\rrbracketsub{{\Del^\dag,0}}=
\llbracket K_{a,b}^{k-1}(X)\rrbracketsub{{\Del^\dag,1}}.$$
\end{proof}

Our penultimate result in this section is a reconfiguration of Lemma \ref{lemma7.2} that facilitates an 
alternative bound equipped with equally weighted exponents. In this context, we recall the definitions 
(\ref{2.8}) and (\ref{2.27}) of $\Del_1(\nu)$ and $\Lam(\Del)$.

\begin{lemma}\label{lemma7.3}
Suppose that $a$ and $b$ are integers with $0\le a<b\le \tet^{-1}$ and $b\ge (r-1)a$. Then one has
$$\llbracket K_{a,b}^{k-1}(X)\rrbracketsub{{\Del,1}}\ll 
\left( (X/M^b)^{\Lam+\del}\right)^{(s-k+1)/s}
\llbracket I_{b,k_0b}^{k-1}(X)\rrbracketsub{{\Del,1}}^{(k-r-1)/s}
\prod_{j=1}^r\llbracket I_{b,k_jb}^{k-1}(X)\rrbracketsub{{\Del,1}}^{1/s},$$
where $\Del=\Del_1(\nu)$ and $\Lam=\Lam(\Del_1(\nu))$.
\end{lemma}

\begin{proof} Write
\begin{equation}\label{7.12}
\Del_\rho=\sum_{m=1}^\rho \frac{(m-1)(k-m-1)}{s-m}.
\end{equation}
Then on recalling (\ref{2.24}) and (\ref{2.25}), we find that
\begin{equation}\label{7.13}
\llbracket I_{b,(k-\rho+j)b}^{k-1}(X)\rrbracketsub{{0,1}}=(X/M^b)^{\Del_\rho}
\llbracket I_{b,(k-\rho+j)b}^{k-1}(X)\rrbracketsub{{\Del_\rho,1}}\quad (0\le j\le \rho)
\end{equation}
and
\begin{equation}\label{7.14}
\llbracket K_{a,b}^{k-1}(X)\rrbracketsub{{0,1}}=(X/M^a)^{\Del_\rho}
\llbracket K_{a,b}^{k-1}(X)\rrbracketsub{{\Del_\rho,1}}.
\end{equation}
By reference to (\ref{6.16}) and (\ref{7.1}) to (\ref{7.3}), it therefore follows from Lemma 
\ref{lemma7.2} that for $0\le \rho\le r$, one has
\begin{align}
\llbracket K_{a,b}^{k-1}(X)\rrbracketsub{{0,1}}\ll &\, (M^{b-a})^{\Del_\rho}
\left( (X/M^b)^{\Lam(0)+\del}\right)^{\tfrac{s-k+1}{s}}\llbracket I_{b,(k-\rho)b}^{k-1}
(X)\rrbracketsub{{0,1}}^{\tfrac{k-\rho-1}{s-\rho}} \notag\\
&\, \times \prod_{j=1}^\rho \llbracket I_{b,(k-\rho+j)b}^{k-1}(X)\rrbracketsub{{0,1}}^{\tfrac{s-k+1}
{(s-\rho+j-1)(s-\rho+j)}}.\label{7.15}
\end{align}
By applying this bound successively for $\rho=r, r-1,\ldots ,0$, we deduce that
\begin{align*}
\llbracket K_{a,b}^{k-1}(X)\rrbracketsub{{0,1}}=&\, 
\llbracket K_{a,b}^{k-1}(X)\rrbracketsub{{0,1}}^{(s-r)/s}
\prod_{\rho=0}^{r-1}\llbracket K_{a,b}^{k-1}(X)\rrbracketsub{{0,1}}^{1/s}\\
\ll &\, (M^{b-a})^{\Del^*/s}
\left( (X/M^b)^{\Lam(0)+\del}\right)^{\tfrac{s-k+1}{s}}\llbracket I_{b,(k-r)b}^{k-1}
(X)\rrbracketsub{{0,1}}^{(k-r-1)/s}\\
&\, \times \prod_{j=1}^r\llbracket I_{b,(k-r+j)b}^{k-1}(X)\rrbracketsub{{0,1}}^{\ome_j/s},
\end{align*}
where
$$\Del^*=(s-r)\Del_r+\sum_{\rho=0}^{r-1}\Del_\rho$$
and
\begin{align*}
\ome_j=&\, \frac{(s-r)(s-k+1)}{(s-r+j-1)(s-r+j)}+\frac{k-r+j-1}{s-r+j}\\
&\, +\sum_{\rho=r-j+1}^{r-1}\frac{s-k+1}{(s-r+j-1)(s-r+j)}.
\end{align*}

We observe first that, as a consequence of (\ref{7.12}), one has
\begin{align*}
\Del^*&\, =(s-r)\sum_{m=1}^r\frac{(m-1)(k-m-1)}{s-m}+\sum_{\rho=0}^{r-1}
\sum_{m=1}^\rho \frac{(m-1)(k-m-1)}{s-m}\\
&\,=(s-r)\sum_{m=1}^r\frac{(m-1)(k-m-1)}{s-m}+\sum_{m=1}^r
\frac{(m-1)(k-m-1)}{s-m}\sum_{\rho=m}^{r-1}1\\
&\,=\sum_{m=1}^r(m-1)(k-m-1).
\end{align*}
Recalling the definition (\ref{2.8}) of $\Del_1(\nu)$, therefore, we conclude that $\Del^*=s\Del$. Also, 
one sees that
$$\ome_j=\frac{s-k+1}{s-r+j}+\frac{k-r+j-1}{s-r+j}=1.$$
Thus we infer that
\begin{align*}
M^{\Del a}\llbracket K_{a,b}^{k-1}(X)\rrbracketsub{{0,1}}\ll &\, M^{\Del b} 
\left( (X/M^b)^{\Lam(0)+\del}\right)^{\tfrac{s-k+1}{s}}\llbracket I_{b,(k-r)b}^{k-1}
(X)\rrbracketsub{{0,1}}^{(k-r-1)/s}\\
&\, \times \prod_{j=1}^r\llbracket I_{b,(k-r+j)b}^{k-1}(X)\rrbracketsub{{0,1}}^{1/s},
\end{align*}
and the conclusion of the lemma follows on recalling the relations (\ref{7.13}) and (\ref{7.14}), though
 now with $\Del_\rho$ replaced by $\Del$.
\end{proof}

Finally, we extract from Lemma \ref{lemma7.2} a conclusion related to that of Lemma \ref{lemma7.3}, but 
one that makes available estimates utilising the full power underlying our methods. In this context, it is 
useful to observe that our analysis of the iteration process in \S9 requires that we work with an integral 
number of variables. The interpolation residing in the next lemma addresses this requirement. Before 
announcing this refinement, we recall the definitions (\ref{2.5}) and (\ref{2.7}) of $\nu_0(r,s)$ and 
$\Del_0(\nu)$, and also the definition (\ref{7.1}) of $\phi_j$.

\begin{lemma}\label{lemma7.4} Let $r$ and $\nu$ be integers with $r\ge 2$ and $0\le \nu\le \nu_0(r,s)$, 
and put
$$\sig=\frac{(\nu_0(r,s)-\nu)(s-r)}{(k-r-1)s}.$$
Suppose that $a$ and $b$ are integers with $0\le a<b\le \tet^{-1}$ and $b\ge (r-1)a$. Then one has
$$\llbracket K_{a,b}^{k-1}(X)\rrbracketsub{{\Del,1}}\ll 
\left( (X/M^b)^{\Lam+\del}\right)^{(s-k+1)/s}
\prod_{j=0}^r\llbracket I_{b,k_jb}^{k-1}(X)\rrbracketsub{{\Del,1}}^{\phi_j(\sig)},$$
where $\Del=\Del_0(\nu)$, $\Lam=\Lam(\Del_0(\nu))$, and
\begin{align}
\phi_0(\sig)&=\frac{(k-r-1)(1-\sig)}{s-r},\label{7.16}\\
\phi_1(\sig)&=\frac{(s-k+1)(1-\sig)}{(s-r)(s-r+1)}+\frac{(k-r)\sig}{s-r+1},\label{7.17}\\
\phi_j(\sig)&=\phi_j\quad (2\le j\le r).\label{7.18}
\end{align}
\end{lemma}

\begin{proof} We adapt the argument of the proof of Lemma \ref{lemma7.3}, making use of the notation 
(\ref{7.12}) and the relations (\ref{7.13}) to (\ref{7.15}). By reference to (\ref{2.4}) and (\ref{2.5}), one 
finds that when $k\ge 4$, one has
\begin{align*}
\sig&\le \frac{s-r}{(k-r-1)s}\sum_{m=1}^r\frac{m(k-m-1)}{s-m}\\
&\le \sum_{1\le m\le (k-1)/2}\frac{k-m-1}{s}+\frac{(k+1)/2}{s}.
\end{align*}
In this way, it is easily verified that when $k\ge 3$, the parameter $\sig$ satisfies $0\le \sig<1$, and 
hence, by applying (\ref{7.15}) for $\rho=r$ and $r-1$, we deduce that
\begin{align*}
\llbracket K_{a,b}^{k-1}(X)\rrbracketsub{{0,1}}=&\, 
\llbracket K_{a,b}^{k-1}(X)\rrbracketsub{{0,1}}^{1-\sig}
\llbracket K_{a,b}^{k-1}(X)\rrbracketsub{{0,1}}^{\sig}\\
\ll &\, (M^{b-a})^{\Del^*}\left( (X/M^b)^{\Lam(0)+\del}\right)^{\tfrac{s-k+1}{s}}
\prod_{j=0}^r\llbracket I_{b,(k-r+j)b}^{k-1}(X)\rrbracketsub{{0,1}}^{\phi_j(\sig)},
\end{align*}
where
$$\Del^*=(1-\sig)\Del_r+\sig \Del_{r-1}.$$
By making use of the definitions (\ref{2.5}) and (\ref{2.7}), we find that
$$\Del^*=\Del_r-\left( \frac{(\nu_0(r,s)-\nu)(s-r)}{(k-r-1)s}\right)
\left( \frac{(r-1)(k-r-1)}{s-r}\right)=\Del_0(\nu).$$
The conclusion of the lemma follows on recalling the relations (\ref{7.13}) and (\ref{7.14}), though in this 
instance we replace $\Del_\rho$ by the value of $\Del$ given in the statement of the lemma.
\end{proof}

\section{The latent monograde process} The estimates supplied by Lemmata \ref{lemma7.2}, 
\ref{lemma7.3} and \ref{lemma7.4} could, in principle, be applied in an iterative manner so as to bound
 $K_{a,b}^{k-1}(X)$ in terms of the $r+1$ mean values $K_{b,k_jb}^{k-1}(X)$ $(0\le j\le r)$, each of
 which could be bounded in terms of $r+1$ new mean values of the shape $K_{b',k_jb'}^{k-1}(X)$ 
$(0\le j\le r)$, and so on. This, indeed, is the strategy proposed in the speculative heuristic argument
 described in \cite[\S11]{Woo2012a}. After $N$ iterations, one then has a bound for $K_{a,b}^{k-1}(X)$
 in terms of $(r+1)^N$ new mean values, and one is left with the task of analysing the consequences of
 this iteration. Since we have yet to take account of the need to condition the mean values
 $I_{b,k_jb}^{k-1}(X)$ $(0\le j\le r)$ occurring as intermediate steps in this process, the complexity of 
this analysis would be formidable indeed. Fortunately, we are able to make use of a simplified analysis by 
focusing attention on just one of the $r+1$ mean values at each stage, this being achieved by applying a 
weighted version of an estimate related to H\"older's inequality (see Lemma \ref{lemma8.1} below). In 
this way, the complicated product of mean values produced by Lemma \ref{lemma7.4} is bounded in 
terms of a sum of mean values, and one may then focus on the single summand which is maximal. Thus, 
it transpires that one may convert the multigrade iteration into a monograde process that loses none of the 
potential of a full-blown analysis.

\begin{lemma}\label{lemma8.1}
Suppose that $z_0,\ldots ,z_r\in \dbC$, and that $\bet_i$ and $\gam_i$ are positive real numbers for 
$0\le i\le r$. Put $\Ome=\bet_0\gam_0+\ldots +\bet_r\gam_r$. Then one has
$$|z_0^{\bet_0}\ldots z_r^{\bet_r}|\le \sum_{i=0}^r|z_i|^{\Ome/\gam_i}.$$
\end{lemma}

\begin{proof} We apply the elementary inequality
$$|Z_0^{\tet_0}\ldots Z_r^{\tet_r}|\le \sum_{i=0}^r|Z_i|^{\tet_0+\ldots +\tet_r}.$$
Thus, on taking $Z_i=z_i^{1/\gam_i}$ and $\tet_i=\bet_i\gam_i$ for $0\le i\le r$, we obtain the bound
$$\prod_{i=0}^r|z_i^{1/\gam_i}|^{\bet_i\gam_i}\le \sum_{i=0}^r(|z_i|^{1/\gam_i})^\Ome .$$
This completes the proof of the lemma.
\end{proof}

Before announcing the lemma that encodes the latent monograde iteration process, we recall the 
definitions (\ref{2.6}) to (\ref{2.8}), put $s_\iota^*=s_\iota(\nu)$, and define 
$\rho_j=\rho_{j,\iota}(k,r,s)$ by
\begin{equation}\label{8.1}
\rho_j=k_js/s_\iota^*\quad (\iota=0,1).
\end{equation}
Note also that, as in all of the work of \S\S2--9, we assume throughout that the parameter $r$ satisfies the 
condition (\ref{2.4}), and also that $s\ge s_\iota^*$.

\begin{lemma}\label{lemma8.2} Let $\iota$ be either $0$ or $1$, and put $\Del=\Del_\iota(\nu)$ and 
$\Lam=\Lam(\Del_\iota (\nu))$. Suppose that $\Lam\ge 0$, and let $a$ and $b$ be integers with 
$$0\le a<b\le (32k\tet)^{-1}\quad \text{and}\quad b\ge (r-1)a.$$
Suppose in addition that there are real numbers $\psi$, $c$ and $\gam$, with
$$0\le c\le (2\del)^{-1}\tet,\quad \gam\ge -b\quad \text{and}\quad \psi\ge 0,$$
such that
\begin{equation}\label{8.2}
X^\Lam M^{\Lam \psi}\ll X^{c\del}M^{-\gam}\llbracket K_{a,b}^{k-1}(X)\rrbracketsub{{\Del,1}}.
\end{equation}
Then, for some integers $j$ and $h$ with $0\le j\le r$ and $0\le h\le 15(k_j-1)b$, one has the upper 
bound
$$X^\Lam M^{\Lam \psi'}\ll X^{c'\del}M^{-\gam'}\llbracket 
K_{b,k_jb+h}^{k-1}(X)\rrbracketsub{{\Del,1}},$$
where
$$\psi'=\rho_j\left(\psi+\left(1-(k-1)/s\right)b\right),\quad c'=\rho_j(c+1),\quad 
\gam'=\rho_j\gam+(\tfrac{4}{3}s-1)h.$$
\end{lemma}

\begin{proof} We begin by establishing the lemma when $\iota=0$, the corresponding argument for 
$\iota=1$ being analogous, though simpler. By hypothesis, we have $X^{c\del}<M^{1/2}$. We therefore 
deduce from the postulated bound (\ref{8.2}) and Lemma \ref{lemma7.4} that
$$X^\Lam M^{\Lam \psi}\ll X^{(c+1)\del}M^{-\gam}(X/M^b)^{\Lam \phi^*}\prod_{j=0}^r
\llbracket I_{b,k_jb}^{k-1}(X)\rrbracketsub{{\Del,1}}^{\phi_j(\sig)},$$
where the exponents $\phi^*$ and $\phi_j(\sig)$ are defined by means of (\ref{7.2}) and (\ref{7.16}) to 
(\ref{7.18}). Thus, on verifying that $\phi_0(\sig)+\phi_1(\sig)=\phi_0+\phi_1$, and then making use of 
(\ref{7.3}), we deduce that
\begin{equation}\label{8.3}
\prod_{j=0}^r\left( X^{-\Lam}\llbracket 
I_{b,k_jb}^{k-1}(X)\rrbracketsub{{\Del,1}}\right)^{\phi_j(\sig)}\gg 
X^{-(c+1)\del}M^{\Lam (\psi+\phi^*b)+\gam}.
\end{equation}

\par We next prepare for our application of Lemma \ref{lemma8.1}, but we first examine a related 
situation. Put $\bet_j=\phi_j$ and $\gam_j=k_j$ for $0\le j\le r$. We write 
$\Ome_r=\bet_0\gam_0+\ldots+\bet_r\gam_r$, and recall the relation (\ref{7.9}). Then
 we find from (\ref{6.16}), (\ref{7.1}) and (\ref{7.2}) that
$$\Ome_r=\frac{k_0(k_0-1)}{s_0}+\sum_{j=1}^r\frac{(s-k+1)k_j}{s_{j-1}s_j}
=\frac{k_{r-1}k_r}{s_r}-\sum_{j=1}^r\frac{k_{j-2}}{s_{j-1}}.$$
Thus, on making the change of variable $m=r-j+1$, and referring once more to (\ref{6.16}), we find that
\begin{align*}
s\Ome_r&=k(k-1)-\sum_{m=1}^r(k-m-1)-\sum_{m=1}^r\frac{m(k-m-1)}{s-m}\\
&=k^2-(r+1)k+\tfrac{1}{2}r(r+3)-\sum_{m=1}^r\frac{m(k-m-1)}{s-m}.
\end{align*}
On writing $\Ome_{r-1}$ for the analogue of $\Ome_r$ in which $r$ is replaced by $r-1$, therefore, we 
find that
$$s\left( (1-\sig)\Ome_r+\sig \Ome_{r-1}\right)=k^2-(r+1)k+\tfrac{1}{2}r(r+3)-\nu^+,$$
where
$$\nu^+=\sum_{m=1}^r\frac{m(k-m-1)}{s-m}-\sig \left( k-r-1+\frac{r(k-r-1)}{s-r}\right) .$$ 
By reference to (\ref{2.5}) and the definition of $\sig$ from Lemma \ref{lemma7.4}, we find that
$$\nu^+=\nu_0(r,s)-\frac{(\nu_0(r,s)-\nu)(s-r)}{s}\left( 1+\frac{r}{s-r}\right) =\nu,$$
and hence it follows from (\ref{2.6}) that
$$s\left( (1-\sig)\Ome_r+\sig \Ome_{r-1}\right)=s_0(\nu).$$

\par With the last equation in hand, we now write $\bet_j=\phi_j(\sig)$ and $\gam_j=k_j$ for 
$0\le j\le r$, and put $\Ome=\bet_0\gam_0+\ldots +\bet_r\gam_r$. Then we find from (\ref{7.1}), 
(\ref{7.2}) and (\ref{7.16}) to (\ref{7.18}) that
$$\Ome =(1-\sig)\Ome_r+\sig \Ome_{r-1}=s_0(\nu)/s.$$
A comparison with (\ref{8.1}) therefore reveals that
$$\Ome/k_j=s_0^*/(sk_j)=1/\rho_j.$$
By wielding Lemma \ref{lemma8.1} against (\ref{8.3}), we thus deduce that
$$\sum_{j=0}^r\left( X^{-\Lam}\llbracket 
I_{b,k_jb}^{k-1}(X)\rrbracketsub{{\Del,1}}\right)^{1/\rho_j}\gg X^{-(c+1)\del}
M^{\Lam (\psi+\phi^*b)+\gam}.$$
Consequently, for some index $j$ with $0\le j\le r$, one has
$$X^{-\Lam}\llbracket I_{b,k_jb}^{k-1}(X)\rrbracketsub{{\Del,1}}\gg
 X^{-\rho_j(c+1)\del}M^{\Lam \rho_j(\psi+\phi^*b)+\rho_j\gam},$$
whence, by reference to (\ref{7.2}), we conclude that
\begin{equation}\label{8.4}
X^\Lam M^{\Lam \psi'}\ll X^{c'\del}M^{-\rho_j\gam}\llbracket 
I_{b,k_jb}^{k-1}(X)\rrbracketsub{{\Del,1}}.
\end{equation}

\par We fix the integer $j$ so that the upper bound (\ref{8.4}) holds, and write $H=15(k_j-1)b$. The 
estimate (\ref{8.4}) comes close to achieving the bound claimed in the conclusion of the lemma, though it
 remains to condition the mean value $I_{b,k_jb}^{k-1}(X)$ so as to replace it with a suitable mean value 
of the form $K_{b,c}^{k-1}(X)$. As a consequence of Lemma \ref{lemma4.2}, there exists an integer $h$ 
with $0\le h<H$ such that
$$I_{b,k_jb}^{k-1}(X)\ll (M^h)^{2s/3}K_{b,k_jb+h}^{k-1}(X)+(M^H)^{-s/4}(X/M^{k_jb})^{2s}
(X/M^b)^{\lam-2s}.$$
Then in view of (\ref{6.16}), we infer from (\ref{2.24}) and (\ref{2.25}) that
$$\llbracket I_{b,k_jb}^{k-1}(X)\rrbracketsub{{\Del,1}}\ll M^\ome \llbracket 
K_{b,k_jb+h}^{k-1}(X)\rrbracketsub{{\Del,1}}+(M^H)^{-s/4}M^{(r-j)b}X^\Lam,$$
where
$$\ome=h+2sh/3-2sh=(1-4s/3)h.$$
Substituting this estimate into (\ref{8.4}), therefore, we deduce that
\begin{equation}\label{8.5}
X^\Lam M^{\Lam \psi'}\ll X^{c'\del}M^{-\gam'}\llbracket 
K_{b,k_jb+h}^{k-1}(X)\rrbracketsub{{\Del,1}}+\Psi,
\end{equation}
where
\begin{equation}\label{8.6}
\Psi=(M^H)^{-s/4}M^{(r-j)b-\rho_j\gam}X^{\Lam+c'\del}.
\end{equation}

\par We now set about analysing the term $\Psi$, and seek to show that it makes a negligible contribution 
in (\ref{8.5}). Observe that from (\ref{2.4}) and (\ref{2.5}), one has
$$\nu_0(r,s)\le \frac{r(r+1)(k-2)/2}{k(k+1)/2}<\tfrac{1}{2}r,$$
so that (\ref{2.6}) delivers the lower bound
$$s_0^*\ge k^2-(r+1)k+\tfrac{1}{2}r(r+2)>(k-\tfrac{1}{2}(r+1))^2-1\ge k.$$
Our ambient hypotheses ensure also that $k_j\ge 2$, and thus $H=15(k_j-1)b\ge 15b$. We therefore 
obtain the lower bound
$$\tfrac{1}{4}Hs\ge 3bs\ge 3kbs/s_0^*.$$
Meanwhile, from (\ref{6.16}) and (\ref{8.1}), one has
\begin{align*}
\rho_j+(r-j)b-\rho_j\gam &\le \left(2k_j+(r-j)s_0^*/s\right) bs/s_0^*\\
&\le 2kbs/s_0^*<\tfrac{1}{4}Hs.
\end{align*}
Since our hypotheses ensure also that
$$c'\del\le \rho_j(c+1)\del \le \rho_j\tet-\del,$$
we conclude from (\ref{8.6}) that
$$\Psi\ll (M^H)^{-s/4}M^{(r-j)b-\rho_j\gam +\rho_j}X^{\Lam-\del}\ll X^{-\del}
\left( X^\Lam M^{\Lam \psi'}\right).$$
The conclusion of the lemma for $\iota=0$ follows by substituting this estimate into (\ref{8.5}).\par

We now turn to the situation with $\iota=1$. We proceed as before, but now deduce from Lemma 
\ref{lemma7.3} that the lower bound (\ref{8.3}) holds with the exponents $\phi_j(\sig)$ in this instance 
modified so that
$$\phi_0(\sig)=(k-r-1)/s\quad \text{and}\quad \phi_j(\sig)=1/s\quad (1\le j\le r).$$
With this modification in hand, and $\bet_j$ and $\gam_j$ defined in the same manner as before, we find 
by reference to (\ref{2.5}) and (\ref{2.6}) that
$$s\Ome=(k-r-1)(k-r)+\sum_{j=1}^r(k-r+j)=s_1^*.$$
Thus, with the modified definitions automatically implied by our shift from $\iota=0$ to $\iota=1$, one
 finds that (\ref{8.4}) holds also in this case. From here we may follow precisely the same argument as in 
the case $\iota=0$, delivering again the conclusion of the lemma in this second case.
\end{proof}

\section{The iterative process} In common with our previous efficient congruencing methods, the
 conclusion of Lemma \ref{lemma8.2} provides the basis for a concentration argument. Thus, if the mean 
value $K_{a,b}^{k-1}(X)$ is significantly larger than its ``expected'' magnitude, then for some index $j$ 
and a suitable non-negative integer $h$, the related mean value $K_{b,k_jb+h}^{k-1}(X)$ exceeds its 
``expected'' magnitude by an even larger margin. By iterating this process, we amplify this excess to the 
point that we obtain a contradiction. Since Lemma \ref{lemma5.1} bounds $J_\grw(X)$ in terms of 
$K_{0,1+h}^{k-1}(X)$, for some $h\in \{0,1,2,3\}$, we are able to infer that $J_\grw(X)$ is very close to 
its ``expected'' magnitude. The main difficulty we face in this paper, as opposed to previous work  
\cite{FW2013, Woo2012a, Woo2013}, is that the modulus amplification factor $k_j$ varies from one 
iteration to the next. Fortunately, with care, our previous analyses may be adapted to accommodate this 
complication. We begin by recalling a crude upper bound for $K_{a,b}^{k-1}(X)$.

\begin{lemma}\label{lemma9.1}
Suppose that $a$ and $b$ are integers with $0\le a<b\le (2\tet)^{-1}$, and let $\Lam=\Lam(\Del)$. Then 
provided that $\Lam\ge 0$, one has
$$\llbracket K_{a,b}^{k-1}(X)\rrbracketsub{{\Del,1}}\ll X^{\Lam+\del}(M^{b-a})^sM^{ka-b}.$$
\end{lemma}

\begin{proof} The desired conclusion follows by the argument applied in the proof of 
\cite[Lemma 5.3]{FW2013}, on noting that $\llbracket K_{a,b}^{k-1}(X)\rrbracketsub{{\Del,1}}=
M^{ka-b}\llbracket K_{a,b}^{k-1}(X)\rrbracketsub{{\Del,0}}$.
\end{proof}

We can now announce a mean value estimate for $J_\grw(X)$ that in many circumstances is a little sharper 
than that recorded in Theorem \ref{theorem1.1}. Since this estimate is likely to be of use in future 
applications, as indeed is the case in \S\S10--12 of this paper, we deliberately opt for a relatively 
transparent form.

\begin{theorem}\label{theorem9.2} Suppose that $k$, $r$ and $s$ are natural numbers with $k\ge 3$,
$$r\le \min \{ k-2,\tfrac{1}{2}(k+1)\}\quad \text{and}\quad s\ge s_0,$$
where
$$s_0=k^2-(r+1)k+\tfrac{1}{2}r(r+3)-\nu_0(r,s)$$
and
$$\nu_0(r,s)=\sum_{m=1}^r\frac{m(k-m-1)}{s-m}.$$
Put
$$\nu=\max\{ k^2-(r+1)k+\tfrac{1}{2}r(r+3)-s,0\}.$$
Then for each $\eps>0$, one has
$$J_{s+k-1}(X)\ll X^{2s+2k-2-\frac{1}{2}k(k+1)+\Del+\eps},$$
where
$$\Del=\sum_{m=1}^r\frac{(m-1)(k-m-1)}{s-m}-\frac{(\nu_0(r,s)-\nu)(r-1)}{s}.$$
\end{theorem}

\begin{proof} Write $\Lam=\Lam(\Del)$, and note that $\Del=\Del_0(\nu)$. We prove that $\Lam\le 0$, 
for the conclusion of the lemma then follows at once from (\ref{2.26}). Recall the definition (\ref{2.6}). 
Then we may suppose also that $s=s_0(\nu)$, for if $t>\grw$, then a trivial estimate delivers the estimate 
$J_t(X)\ll X^{2(t-\grw)}J_\grw(X)$, and thus the desired conclusion follows from the upper bound 
provided by the theorem in the case $s=s_0(\nu)$. Assume then that $\Lam\ge 0$, for otherwise there is 
nothing to prove. We begin by noting that as a consequence of Lemma \ref{lemma5.1}, one finds from 
(\ref{2.23}) and (\ref{2.25}) that there exists an integer $h_{-1}\in\{0,1,2,3\}$ such that
\begin{align*}
\llbracket J_\grw(X)\rrbracketsub{{\Del}}&\ll (M^{h_{-1}})^{-4s/3}\llbracket 
K_{0,1+h_{-1}}^{k-1}(X)\rrbracketsub{{\Del,0}}\\
&\ll M^{1-(4s/3-1)h_{-1}}\llbracket K_{0,1+h_{-1}}^{k-1}(X)\rrbracketsub{{\Del,1}}.
\end{align*}
We therefore deduce from (\ref{2.26}) that
\begin{equation}\label{9.1}
X^\Lam \ll X^\del \llbracket J_\grw(X)\rrbracketsub{{\Del}}\ll X^\del M(M^{h_{-1}})^{-(4s/3-1)}
\llbracket K_{0,1+h_{-1}}^{k-1}(X)
\rrbracketsub{{\Del,1}} .
\end{equation}

\par Next we define sequences $(\kap_n)$, $(h_n)$, $(a_n)$, $(b_n)$, $(c_n)$, $(\psi_n)$ and 
$(\gam_n)$, for $0\le n\le N$, in such a way that
\begin{equation}\label{9.2}
k-r\le \kap_{n-1}\le k,\quad 0\le h_{n-1}\le 15(\kap_{n-1}-1)b_{n-1}\quad (n\ge 1),
\end{equation}
and
\begin{equation}\label{9.3}
X^\Lam M^{\Lam \psi_n}\ll X^{c_n\del}M^{-\gam_n}\llbracket 
K_{a_n,b_n}^{k-1}(X)\rrbracketsub{{\Del,1}}.
\end{equation}
Given a fixed choice for the sequences $(\kap_n)$ and $(h_n)$, the remaining sequences are defined by 
means of the relations
\begin{equation}\label{9.4}
a_{n+1}=b_n,\quad b_{n+1}=\kap_nb_n+h_n,
\end{equation}
\begin{equation}\label{9.5}
c_{n+1}=\kap_n(c_n+1),
\end{equation}
\begin{equation}\label{9.6}
\psi_{n+1}=\kap_n\psi_n+\kap_n(1-(k-1)/s)b_n,
\end{equation}
\begin{equation}\label{9.7}
\gam_{n+1}=\kap_n\gam_n+(4s/3-1)h_n.
\end{equation}
We put
$$\kap_{-1}=k,\quad a_0=0,\quad b_{-1}=1,\quad b_0=1+h_{-1},$$
$$\psi_0=0,\quad c_0=1,\quad \gam_0=(4s/3-1)h_{-1}-1,$$
so that both (\ref{9.2}) and (\ref{9.3}) hold with $n=0$ as a consequence of our initial choice of 
$h_{-1}$ together with (\ref{9.1}). We prove by induction that for each non-negative integer $n$ with 
$n<N$, the sequences $(\kap_m)_{m=0}^n$ and $(h_m)_{m=-1}^n$ may be chosen in such a way that
\begin{equation}\label{9.8}
0\le a_n<b_n\le (32\kap_n\tet)^{-1},\quad b_n\ge (r-1)a_n,
\end{equation}
\begin{equation}\label{9.9}
\psi_n\ge 0,\quad \gam_n\ge -b_n,\quad 0\le c_n\le (2\del)^{-1}\tet,
\end{equation}
and so that (\ref{9.2}) and (\ref{9.3}) both hold with $n$ replaced by $n+1$.\par

Let $0\le n<N$, and suppose also that (\ref{9.2}) and (\ref{9.3}) both hold for the index $n$. We have 
already shown such to be the case for $n=0$. We observe first that the relation (\ref{9.4}) demonstrates 
that $b_n>a_n$ for all $n$. Moreover, since our hypotheses on $r$ ensure that $\kap_n\ge k-r\ge r-1$, it 
follows from (\ref{9.4}) that one has $b_n\ge (r-1)a_n$. Also, from (\ref{9.2}) and (\ref{9.4}), we find 
that
$$b_n\le 4\cdot 16^n\kap_0\ldots \kap_{n-1}\le 4(16k)^n,$$
whence, by invoking (\ref{2.9}), we see that for $0\le n<N$ one has
$$b_n\le (32k\tet)^{-1}\le (32\kap_n\tet)^{-1}.$$
It is apparent from (\ref{9.5}) and (\ref{9.6}) that $c_n$ and $\psi_n$ are non-negative for all $n$. 
Observe also that since $\kap_m\le k$, then by iterating (\ref{9.5}) we obtain the bound
\begin{equation}\label{9.10}
c_n\le k^n+k\Bigl( \frac{k^n-1}{k-1}\Bigr)\le 3k^n\quad (n\ge 0),
\end{equation}
and by reference to (\ref{2.9}) we see that $c_n\le (2\del)^{-1}\tet$ for $0\le n<N$.\par

In order to bound $\gam_n$, we begin by noting from (\ref{9.4}) that for $m\ge 1$, one has
$$h_m=b_{m+1}-\kap_mb_m\quad \text{and}\quad a_m=b_{m-1}.$$
Then it follows from (\ref{9.7}) that for $m\ge 1$ one has
$$\gam_{m+1}-(4s/3-1)b_{m+1}=\kap_m\left(\gam_m-(4s/3-1)b_m\right).$$
By iterating this relation, we deduce that for $m\ge 1$, one has
$$\gam_m=(4s/3-1)b_m+\kap_0\ldots \kap_{m-1}(\gam_0-(4s/3-1)b_0).$$
Recall next that $b_0=1+h_{-1}$ and $\gam_0=(4s/3-1)h_{-1}-1$. Then we discern that
\begin{equation}\label{9.11}
\gam_m=(4s/3-1)b_m-(4s/3)\kap_0\ldots \kap_{m-1}\quad (m\ge 1).
\end{equation}
Finally, we find from (\ref{9.4}) that for $m\ge 0$ one has $b_{m+1}\ge \kap_mb_m$, so that an 
inductive argument yields the lower bound $b_m\ge \kap_0\ldots \kap_{m-1}$ for $m\ge 0$. Hence we 
deduce that
$$\gam_m\ge \tfrac{4}{3}s(b_m-\kap_0\ldots \kap_{m-1})-b_m\ge -b_m.$$
Assembling this conclusion together with those of the previous paragraph, we have shown that both 
(\ref{9.8}) and (\ref{9.9}) hold for $0\le n\le N$.\par

At this point in the argument, we may suppose that (\ref{9.3}), (\ref{9.8}) and (\ref{9.9}) hold for the 
index $n$. An application of Lemma \ref{lemma8.2} therefore reveals that there exist integers $\kap_n$ 
and $h_n$ satisfying the constraints implied by (\ref{9.2}) with $n$ replaced by $n+1$, for which the 
upper bound (\ref{9.3}) holds also with $n$ replaced by $n+1$. This completes the inductive step, so that 
in particular the upper bound (\ref{9.3}) holds for $0\le n\le N$.\par

We now exploit the bound just established. Since we have the upper bound 
$b_N\le 4(16k)^N\le (2\tet)^{-1}$, it is a consequence of Lemma \ref{lemma9.1} that
\begin{equation}\label{9.12}
\llbracket K_{a_N,b_N}^{k-1}(X)\rrbracketsub{{\Del,1}}\ll X^{\Lam+\del}
(M^{b_N-b_{N-1}})^sM^{kb_{N-1}-b_N}.
\end{equation}
By combining (\ref{9.3}) with (\ref{9.11}) and (\ref{9.12}), we obtain the bound
\begin{align}
X^\Lam M^{\Lam\psi_N}&\ll X^{\Lam+(c_N+1)\del}
M^{kb_{N-1}-b_N+s(b_N-b_{N-1})-\gam_N}\notag \\
&\ll X^{\Lam+(c_N+1)\del}M^{-\frac{1}{3}sb_N-(s-k)b_{N-1}+\frac{4}{3}s\kap_0\ldots \kap_{N-1}}.
\label{9.13}
\end{align}
Meanwhile, an application of (\ref{9.10}) in combination with (\ref{2.9}) shows that 
$X^{(c_N+1)\del }<M$. We therefore deduce from (\ref{9.13}) and our previous lower bound 
$b_N\ge \kap_0\ldots \kap_{N-1}$ that
$$\Lam \psi_N\le \tfrac{4}{3}s\kap_0\ldots \kap_{N-1}-\tfrac{1}{3}b_Ns\le s\kap_0\ldots \kap_{N-1}.$$
Temporarily, we write $\chi=(k-1)/s$. Then a further application of the lower bound 
$b_n\ge \kap_0\ldots \kap_{n-1}$ leads from (\ref{9.6}) to the relation
$$\psi_{n+1}=\kap_n\psi_n+\kap_n(1-\chi)b_n\ge \kap_n\psi_n+
\kap_n(1-\chi)\kap_0\ldots \kap_{n-1},$$
whence, by an inductive argument, one finds that
$$\psi_N\ge N(1-\chi)\kap_0\ldots \kap_{N-1}.$$
Thus we deduce that
$$\Lam\le \frac{s\kap_0\ldots \kap_{N-1}}{N(1-1/\chi)\kap_0\ldots \kap_{N-1}}=
\frac{s(1-1/\chi)^{-1}}{N}.$$
Since we are at liberty to take $N$ as large as we please in terms of $s$ and $k$, we are forced to
 conclude that $\Lam\le 0$. In view of our opening discussion, this completes the proof of the theorem.
\end{proof}

The proof of Theorem \ref{theorem1.1} follows by precisely the same argument as that employed to 
establish Theorem \ref{theorem9.2}. We have merely to adjust the choice of parameters so that 
$s_0(\nu)$ and $\Del$ are replaced by
$$s_1(0)=k^2-(r+1)k+\tfrac{1}{2}r(r+3)\quad \text{and}\quad 
\Del=\frac{1}{s}\sum_{m=1}^r(m-1)(k-m-1).$$
The argument of the proof then applies just as before, and when $s\ge s_1(0)$ one obtains the bound
$$J_{s+k-1}(X)\ll X^{2s+2k-2-\frac{1}{2}k(k+1)+\Del+\eps},$$
where, following a modest computation, one finds that
$$\Del=\frac{r(r-1)(3k-2r-5)}{6s}.$$
The conclusion of Theorem \ref{theorem1.1} follows on replacing $s$ in the bound just described by 
$s-k+1$ in the statement of the theorem.

\section{The asymptotic formula in Waring's problem} Our first applications of the improved mean value 
estimates supplied by Theorems \ref{theorem1.1} and \ref{theorem9.2} concern the asymptotic formula 
in Waring's problem. In this section we establish Theorems \ref{theorem1.3} and \ref{theorem1.4}, as 
well as a number of auxiliary estimates of use in related topics. In this context, we define the exponential 
sum $g(\alp)=g_k(\alp;X)$ by
$$g_k(\alp;X)=\sum_{1\le x\le X}e(\alp x^k).$$
Also, we define the set of minor arcs $\grm=\grm_k$ to be the set of real numbers $\alp\in [0,1)$ 
satisfying the property that, whenever $a\in \dbZ$ and $q\in \dbN$ satisfy $(a,q)=1$ and 
$|q\alp-a|\le (2k)^{-1}X^{1-k}$, then $q>(2k)^{-1}X$. We begin by applying the methods of 
\cite{Woo2012b} to derive a mean value estimate restricted to minor arcs.\par

The introduction of some additional notation eases our exposition. We define exponents $\nu_\iota^*(r,s)$ 
and $\Del_\iota^*(r,s)$ consistent with the definitions (\ref{2.5}), (\ref{2.7}) and (\ref{2.8}), save that we 
shift the parameter $s$ by $k-1$. Thus, we put
$$\nu_0^*(r,s)=\sum_{m=1}^r\frac{m(k-m-1)}{s-k-m+1}\quad \text{and}\quad \nu^*_1(r,s)=0.$$
We then take $\nu$ to be an integer with $0\le \nu\le \nu_\iota^*(r,s)$, and put
\begin{align}
\Del_0^*(r,s;\nu)&=\sum_{m=1}^r\frac{(m-1)(k-m-1)}{s-k-m+1}-
\frac{(\nu_0^*(r,s)-\nu)(r-1)}{s-k+1},
\label{10.1}\\ 
\Del_1^*(r,s;\nu)&=\sum_{m=1}^r \frac{(m-1)(k-m-1)}{s-k+1}.\label{10.2}
\end{align}

Our first result of this section provides a mean value estimate restricted to minor arcs of use in many 
applications.

\begin{theorem}\label{theorem10.1} Let $\iota$ be either $0$ or $1$. Suppose that $r$, $s$ and $k$ are 
integers with $k\ge 3$,
\begin{equation}\label{10.3}
1\le r\le \min \{ k-2,\tfrac{1}{2}(k+1)\},
\end{equation}
and
\begin{equation}\label{10.4}
s\ge k^2-rk+\tfrac{1}{2}r(r+3)-1-\nu^*_\iota(r,s).
\end{equation}
Put
$$\nu=\max\{ k^2-rk+\tfrac{1}{2}r(r+3)-1-s,0\}.$$
Then for each $\eps>0$, one has
$$\int_\grm |g_k(\alp;X)|^{2s}\d\alp \ll X^{2s-k-1+\Del+\eps},$$
where $\Del=\Del^*_\iota(r,s;\nu)$.
\end{theorem}

\begin{proof} According to \cite[Theorem 2.1]{Woo2012b}, one has
$$\int_\grm|g_k(\alp;X)|^{2s}\d\alp \ll X^{\frac{1}{2}k(k-1)-1}(\log X)^{2s+1}J_{s,k}(2X).$$
By combining Theorems \ref{theorem1.1} and \ref{theorem9.2}, it follows that when $s$ satisfies the 
lower bound (\ref{10.4}), one has 
$$J_{s,k}(2X)\ll X^{2s-\frac{1}{2}k(k+1)+\Del+\eps},$$
and the conclusion of the theorem now follows.
\end{proof}

The dependence on $s$ of $\nu_0^*(r,s)$ suggests that the lower bound (\ref{10.4}) may be difficult to 
interpret. However, since $s$ exceeds $\tfrac{1}{2}k(k+1)$ and $r\le \tfrac{1}{2}(k+1)$, a crude 
computation confirms that $\nu_0^*(r,s)\le k$. Put
$$s_0=k^2-rk+\tfrac{1}{2}r(r+3)-1.$$
In practice one may check successively for the largest integral value of $\nu$ with $0\le \nu\le k$ for which 
$\nu_0^*(r,s_0-\nu)\ge \nu$. This isolates the largest integer $\nu$ for which (\ref{10.4}) holds with 
$s=s_0-\nu$. As we have noted, this maximal value of $\nu$ is no larger than $k$, and so this is not 
particularly expensive computationally.\par

The special case of Theorem \ref{theorem10.1} with $r=1$ merits particular attention.

\begin{corollary}\label{corollary10.2}
Suppose that $s\ge k^2-k+1$. Then for each $\eps>0$, one has
$$\int_\grm|g_k(\alp;X)|^{2s}\d\alp \ll X^{2s-k-1+\eps}.$$
\end{corollary}

The mean value over major arcs $\grM=[0,1)\setminus \grm$ corresponding to that bounded in this 
corollary has order of magnitude $X^{2s-k}$. Thus, as is clear already in \cite{Woo2012b}, estimates of 
the type provided by Corollary \ref{corollary10.2} may be employed in applications as powerful substitutes 
for estimates of Weyl type.

\par We apply these bounds so as to handle the minor arc contribution in Waring's problem, beginning 
with a sketch of the arguments required for smaller values of $k$. For each natural number $k$, one 
begins by computing permissible exponents $\Del_{s,k}$ having the property that
\begin{equation}\label{10.5}
J_{s,k}(X)\ll X^{2s-\frac{1}{2}k(k+1)+\Del_{s,k}+\eps}.
\end{equation}
It is apparent from Corollary \ref{corollary1.2} that one may take $\Del_{s,k}=0$ for $s\ge k^2-k+1$, 
and thus we may concentrate on the interval $1\le s\le k(k-1)$. Consider each integer $s$ in this interval in 
turn. For each integer $r$ satisfying (\ref{10.3}), one may check (for $\iota \in \{0,1\}$) whether the 
lower bound (\ref{10.4}) is satisfied or not. If this lower bound is satisfied, then the exponent 
$\Del_\iota^*(r,s;\nu)$ given by (\ref{10.1}) or (\ref{10.2}) is permissible. As a preliminary value, one 
takes $\Del^*_{s,k}$ to be the least of these permissible exponents $\Del_\iota^*(r,s;\nu)$ as one runs 
through the available choices for $r$ and $\iota$. Next, by applying H\"older's inequality to (\ref{2.3}), it 
is apparent that whenever $s_1$ and $s_2$ are integers with
$$1\le s_1\le s\le s_2\le k^2-k+1,$$
then the upper bound (\ref{10.5}) holds with
$$\Del_{s,k}=\frac{(s-s_1)\Del^*_{s_2,k}+(s_2-s)\Del^*_{s_1,k}}{s_2-s_1}.$$
For each integer $s$ with $1\le s\le k^2-k+1$, therefore, one may linearly interpolate in this manner 
amongst all possible choices of $s_1$ and $s_2$ so as to obtain the smallest available value of 
$\Del_{s,k}$. It is this exponent $\Del_{s,k}$ that we now fix, and use computationally in what follows.
 We note that if one is content to make use of potentially non-optimal conclusions, then one has the 
alternative option of applying Theorems \ref{theorem1.1} and \ref{theorem9.2} for the specific values of 
$s$ given by integral choices of $r$ with $1\le r\le \tfrac{1}{2}(k+1)$.\par

We are now equipped to negotiate the details of our analysis of the asymptotic formula in Waring's 
problem. We employ two strategies, the first of which interpolates between the minor arc estimate supplied 
by Theorem \ref{theorem10.1}, and that offered by Hua's lemma (see \cite[Lemma 2.5]{Vau1997}). Given 
natural numbers $k$ and $t$ with $k\ge 3$ and $1\le t\le k^2-k+1$, define the positive number 
$s_0(j)=s_0(k,t,j)$ by means of the relation
$$s_0(k,t,j)=2t-\frac{(1-\Del_{t,k})(2t-2^{j+1})}{k-j-\Del_{t,k}},$$
and then put
\begin{equation}\label{10.6}
s_1(k)=\min_{\substack{1\le t\le k^2-k+1\\ \Del_{t,k}<1}}
\min_{\substack{0\le j\le k-2\\ 2^j<t}}s_0(k,t,j).
\end{equation}

\begin{lemma}\label{lemma10.3} Suppose that $k$ is a natural number with $k\ge 3$. Then
$$\int_0^1|g_k(\alp;X)|^{s_1(k)}\d\alp \ll X^{s_1(k)-k+\eps}.$$
Moreover, when $s$ is a real number with $s>s_1(k)$, there exists a positive number $\del=\del(k,s)$ 
with the property that
$$\int_\grm|g_k(\alp;X)|^s\d\alp \ll X^{s-k-\del}.$$
\end{lemma}

\begin{proof} We first establish the second conclusion of the lemma. Let the parameters $j$ and $t$ 
correspond to the minimum implict in (\ref{10.6}). Then the second estimate claimed in the lemma is 
immediate from \cite[Theorem 2.1]{Woo2012b} when $s\ge 2t$, since we have $\Del_{t,k}<1$. Here, if 
necessary, we make use of the trivial estimate $|g_k(\alp;X)|\le X$. Indeed, the latter theorem shows that
\begin{equation}\label{10.7}
\int_\grm |g(\alp)|^{2t}\d\alp \ll X^{\frac{1}{2}k(k-1)-1+\eps}J_{t,k}(2X)\ll 
X^{2t-k+\Del_{t,k}-1+2\eps}.
\end{equation}
We suppose therefore that $s_1(k)<s\le 2t$, and we put $\tau=s-s_1(k)$. Then by H\"older's inequality, 
one has
$$\int_\grm|g(\alp)|^s\d\alp\le \Bigl( \int_\grm |g(\alp)|^{2t}\d\alp \Bigr)^a
\Bigl( \int_0^1|g(\alp)|^{2^{j+1}}\d\alp \Bigr)^b,$$
where
$$a=\frac{s-2^{j+1}}{2t-2^{j+1}}\quad \text{and}\quad b=\frac{2t-s}{2t-2^{j+1}}.$$
An application of Theorem \ref{theorem10.1}, in the guise of the estimate (\ref{10.7}), in combination 
with Hua's lemma (see \cite[Lemma 2.5]{Vau1997}) therefore yields the bound
$$\int_\grm|g(\alp)|^s\d\alp \ll X^\eps (X^{2t-k-1+\Del_{t,k}})^a(X^{2^{j+1}-j-1})^b\ll 
X^{s-k-\ome+\eps},$$
where $\ome=a(1-\Del_{t,k})-(k-j-1)b$. A modicum of computation reveals that
\begin{align*}
\ome&=\frac{(k-j-\Del_{t,k})(s-2t)+(1-\Del_{t,k})(2t-2^{j+1})}{2t-2^{j+1}}\\
&=\frac{(k-j-\Del_{t,k})(s-s_1(k))}{2t-2^{j+1}}\ge \tau/(2k^2),
\end{align*}
and consequently the second conclusion of the lemma follows with $\del=\tau/(4k^2)$.\par

When $s=s_1(k)$, the above discussion shows that
$$\int_\grm |g(\alp)|^s\d\alp \ll X^{s-k+\eps}.$$
But on writing $\grM=[0,1)\setminus \grm$, the methods of \cite[Chapter 4]{Vau1997} confirm that 
whenever $s\ge k+2$, one has
$$\int_\grM |g(\alp)|^s\d\alp \ll X^{s-k}.$$
The first conclusion of the lemma follows by combining these two estimates.
\end{proof}

The argument following the proof of \cite[Lemma 3.1]{Woo2012b} may now be adapted to deliver the 
upper bound contained in the following lemma.

\begin{lemma}\label{lemma10.4}
When $k\ge 3$, define $s_1(k)$ as in equation $(\ref{10.6})$. Then one has 
$\Gtil(k)\le \lfloor s_1(k)\rfloor+1$.
\end{lemma}

This upper bound may of course be made explicit for smaller values of $k$. By using a na\"ive computer 
program to optimise the choice of parameters, one obtains the values of $s_1(k)$ reported in Table 2 
below. Here, we have rounded up in the final decimal place reported. The conclusion of Theorem 
\ref{theorem1.4} now follows by inserting the bounds for $s_1(k)$ supplied by Table 2 into Lemma 
\ref{lemma10.4}. When $k=3$ and $4$, the bounds for $s_1(k)$ supplied by Table 2 may be compared 
with the bounds for $\Gtil(k)$ available from Vaughan's refinements \cite{Vau1986a, Vau1986b} of Hua's 
work. The latter work supplies bounds in a sense tantamount to $s_1(3)\le 8$ and $s_1(4)\le 16$. Thus 
our present work, while coming close to these bounds, nonetheless fails the cigar test.\par

$$\boxed{\begin{matrix} k&3&4&5&6&7&8\\
s_1(k)&9.000&16.311&27.413&42.710&60.799&82.023\end{matrix}}$$
$$\boxed{\begin{matrix}k&9&10&11&12&13&14\\
s_1(k)&106.492&133.724&164.453&198.448&235.389&275.661\end{matrix}}$$
$$\boxed{\begin{matrix}k&15&16&17&18&19&20\\
s_1(k)&319.462&367.221&417.870&472.973&529.938&591.528\end{matrix}}$$
\vskip.2cm
\begin{center}\text{Table 2: Upper bounds for $s_1(k)$ described in equation (\ref{10.6}).}\end{center}
\vskip.1cm
\noindent 

For concreteness, we note that reasonable bounds may be computed by hand with relative ease. Thus a 
good approximation to the bound for $s_1(5)$ recorded in Table 2 derives from the permissible exponent 
$\Del_{18,5}=\tfrac{2}{7}$ that stems from Theorem \ref{theorem1.1} with $r=3$ and $k=5$, and then 
the exponent
$$s_0(5,18,3)=36-\frac{(1-\tfrac{2}{7})(36-16)}{5-3-\tfrac{2}{7}}=27\tfrac{2}{3}$$
that determines $s_1(5)$ by means of (\ref{10.6}). Similarly, one finds that the permissible exponent 
$\Del_{26,6}=\tfrac{1}{3}$ is made available by Theorem \ref{theorem1.1} with $r=3$ and $k=6$, and 
then the exponent
$$s_0(6,26,3)=52-\frac{(1-\tfrac{1}{3})(52-16)}{6-3-\tfrac{1}{3}}=43$$
determines an approximation to $s_1(6)$ by means of (\ref{10.6}).\par

For a clean, easy to state upper bound for $\Gtil(k)$ valid for $k\ge 5$, one may proceed as follows. First, 
apply Corollary \ref{corollary1.2} to obtain the permissible exponent $\Del_{t,k}=0$ with $t=k^2-k+1$. 
One then finds from (\ref{10.6}) that
\begin{align*}
s_1(k)\le s_0(k,k^2-k+1,4)&=2(k^2-k+1)-\frac{2(k^2-k+1)-32}{k-4}\\
&=2k^2-4k-4+\frac{6}{k-4},
\end{align*}
so that $s_1(k)<2k^2-4k-2$ whenever $k\ge 8$. Consequently, by reference to Lemma \ref{lemma10.4}, 
one obtains the following upper bound on $\Gtil(k)$.

\begin{corollary}\label{corollary10.5}
Whenever $k\ge 5$, one has $\Gtil(k)\le 2k^2-4k-2$.
\end{corollary}

For larger values of $k$, one may employ the methods of \cite[\S8]{FW2013} in order to improve on the 
bound given in Corollary \ref{corollary10.5}. The statement of our most general conclusion requires a 
little preparation. Let $\Del_{s,k}$ $(s\in \dbN)$ be the exponents defined in the discussion following 
(\ref{10.5}). For each $v\in \dbN$, we define
$$\Del_{v,k}^+=\min\{ \Del_{v,k}-1,\Del_{v,k-1}\}.$$
When $1\le t\le k^2-k+1$, we now define the positive number $u_0(k,t,v,w)$ by means of the relation
$$u_0(k,t,v,w)=2t-\frac{(1-\Del_{t,k})(2t-2v-w(w-1))}{1-\Del_{t,k}+\Del^+_{v,k}/w}.$$
We then put
\begin{equation}\label{10.8}
u_1(k)=\min_{\substack{1\le t\le k^2-k+1\\ \Del_{t,k}<1}}\underset{2v+w(w-1)<2t}
{\min_{1\le w\le k-1}\min_{v\ge 1}}\, u_0(k,t,v,w).
\end{equation}
We begin by announcing an analogue of Lemma \ref{lemma10.3} useful for intermediate and larger values 
of $k$.

\begin{lemma}\label{lemma10.6} Let $k$ be a natural number with $k\ge 3$, and suppose that $s$ is a 
real number with $s>u_1(k)$. Then there exists a positive number $\del=\del(k,s)$ with the property that
$$\int_\grm|g_k(\alp;X)|^s\d\alp \ll X^{s-k-\del}.$$
\end{lemma}

\begin{proof} Let the parameters $v$, $w$ and $t$ correspond to the minimum implicit in (\ref{10.8}). 
Then in view of the implicit hypothesis $\Del_{t,k}<1$, we find just as in the proof of Lemma 
\ref{lemma10.3} that the desired conclusion is an immediate consequence of 
\cite[Theorem 2.1]{Woo2012b} when $s\ge 2t$, on making use of the trivial estimate 
$|g_k(\alp;X)|\le X$. We suppose therefore that $u_1(k)<s\le 2t$, and we put $\tau=s-u_1(k)$. Then by 
H\"older's inequality, one has
$$\int_\grm|g(\alp)|^s\d\alp\le \Bigl( \int_\grm |g(\alp)|^{2t}\d\alp \Bigr)^a
\Bigl( \int_0^1|g(\alp)|^{2v+w(w-1)}\d\alp \Bigr)^{1-a},$$
where
$$a=\frac{s-2v-w(w-1)}{2t-2v-w(w-1)}.$$

\par We next apply Theorem \ref{theorem10.1}, as embodied in (\ref{10.7}), and then wield the 
estimate supplied by \cite[Theorem 8.5]{FW2013}, thus obtaining the bound
\begin{equation}\label{10.9}
\int_\grm |g(\alp)|^s\d\alp \ll X^\eps \left( X^{2t-k-1+\Del_{t,k}}\right)^a 
\left( X^{2v+w(w-1)-k+\Del_{v,k}^+/w}\right)^{1-a}.
\end{equation}
Since we may suppose that
$$s>u_0(k,t,v,w)=\frac{2t\Del_{v,k}^++w(1-\Del_{t,k})(2v+w(w-1))}{w(1-\Del_{t,k})+\Del_{v,k}^+},
$$
we see that
$$a(1-\Del_{t,k})>(1-a)\Del_{v,k}^+/w,$$
and thus the conclusion of the lemma follows at once from (\ref{10.9}).
\end{proof}

The argument following the proof of \cite[Lemma 3.1]{Woo2012b} may be adapted on this occasion to
 give the following upper bound for $\Gtil(k)$.

\begin{lemma}\label{lemma10.7}
When $k\ge 3$, define $u_1(k)$ as in equation $(\ref{10.8})$. Then one has 
$\Gtil(k)\le \lfloor u_1(k)\rfloor +1$.
\end{lemma}

It would appear that Lemma \ref{lemma10.7} yields superior bounds for $\Gtil(k)$ as compared to 
Lemma \ref{lemma10.4} only for $k$ exceeding $25$ or thereabouts. However, for large values of $k$ 
one obtains substantial quantitative improvements on previous bounds for $\Gtil(k)$, these being reported 
in Theorem \ref{theorem1.3}. Suppose that $k$ is a large natural number, and let $\bet$ be a positive 
parameter to be determined in due course. We take
$$r=\lfloor \tfrac{1}{2}(k+1)\rfloor ,\quad w=\lfloor \bet k\rfloor,\quad \text{and}\quad 
v=k^2-rk+\tfrac{1}{2}r(r+3).$$
Then one has
$$\tfrac{5}{8}k^2\le v\le \tfrac{5}{8}k^2+k,$$
so by Theorem \ref{theorem1.1} one finds that the exponent $\Del_{v,k}$ is permissible, where
$$\Del_{v,k}\le \frac{r(r-1)(2k-5)}{6(\tfrac{5}{8}k^2-k+1)}=\tfrac{2}{15}k+O(1).$$
Also, by taking $t=k^2-k+1$, one sees from Corollary \ref{corollary1.2} that $\Del_{t,k}=0$. Then we 
deduce from (\ref{10.8}) that
$$u_1(k)\le 2(k^2-k+1)-\frac{2k^2-2(\tfrac{5}{8}k^2)-(\bet k)^2+O(k)}
{1+\tfrac{2}{15}k/(\bet k)+O(1/k)}.$$
It follows that
\begin{align*}
u_1(k)/(2k^2)&\le 1-\frac{\bet(\frac{3}{8}-\tfrac{1}{2}\bet^2)}{\bet+\frac{2}{15}}+O(1/k)\\
&=\frac{60\bet^3+75\bet+16}{120\bet+16}+O(1/k).\end{align*}
A modest computation confirms that the optimal choice for the parameter $\bet$ is $\xi$, where $\xi$ is 
the real root of the polynomial equation $20\xi^3=1-4\xi^2$. With this choice for $\bet$, one finds that
$$u_1(k)\le \left(\frac{19+75\xi-12\xi^2}{8+60\xi}\right)k^2+O(k).$$
The conclusion of Theorem \ref{theorem1.3} is now immediate from Lemma \ref{lemma10.7}.

\par We finish by noting that the proof of \cite[Theorem 4.2]{Woo2012b} may be adapted in the obvious 
manner so as to establish that when $s>\min\{ s_1(k),u_1(k)\}$, then the anticipated asymptotic formula 
holds for the number of integral solutions of the diagonal equation
$$a_1x_1^k+\ldots +a_sx_s^k=0,$$
with $|\bfx|\le B$. Here, the coefficients $a_i$ $(1\le i\le s)$ are fixed integers. Similar improvements may 
be wrought in upper bounds for $\Gtil^+(k)$, the least number of variables required to establish that the 
anticipated asymptotic formula in Waring's problem holds for almost all natural numbers $n$. Thus, one 
may adapt the methods of \cite[\S5]{Woo2012b} to show that
$$\Gtil^+(k)\le 1+\min \{ \lfloor \tfrac{1}{2}s_1(k)\rfloor ,\lfloor \tfrac{1}{2}u_1(k)\rfloor \}.$$
In this way, one finds that for large values of $k$, one has $\Gtil^+(k)\le 0.772k^2$, and further that 
for $5\le k\le 20$, one has $\Gtil^+(k)\le H^+(k)$, where $H^+(k)$ is given in Table 3 below.

$$\boxed{\begin{matrix} k&5&6&7&8&9&10&11&12\\
H^+(k)&14&22&31&42&54&67&83&100\end{matrix}}$$
$$\boxed{\begin{matrix}k&13&14&15&16&17&18&19&20\\
H^+(k)&118&138&160&184&209&237&265&296\end{matrix}}$$
\vskip.2cm
\begin{center}\text{Table 3: Upper bounds for $H^+(k)$.}\end{center}
\vskip.1cm
\noindent 

\section{Estimates of Weyl type} In this section we briefly discuss some applications of the mean value estimates 
supplied by Theorems \ref{theorem1.1} and \ref{theorem9.2} to analogues of Weyl's inequality. Our first 
conclusion has the merit of being simple to state, and improves on \cite[Theorem 11.1]{Woo2013} for 
$k\ge 4$. We recall the definition of $f_k(\bfalp;X)$ from (\ref{2.2}).

\begin{theorem}\label{theorem11.1} Let $k$ be an integer with $k\ge 4$, and let $\bfalp\in \dbR^k$.
 Suppose that there exists a natural number $j$ with $2\le j\le k$ such that, for some $a\in \dbZ$ and
 $q\in \dbN$ with $(a,q)=1$, one has $|\alp_j-a/q|\le q^{-2}$ and $q\le X^j$. Then one has
$$f_k(\bfalp;X)\ll X^{1+\eps}(q^{-1}+X^{-1}+qX^{-j})^{\sig(k)},$$
where $\sig(k)^{-1}=2(k^2-3k+3)$.
\end{theorem}

\begin{proof} Under the hypotheses of the statement of the theorem, we find that 
\cite[Theorem 5.2]{Vau1997} shows that for $s\in \dbN$, one has
$$f_k(\bfalp;X)\ll (J_{s,k-1}(2X)X^{\frac{1}{2}k(k-1)}(q^{-1}+X^{-1}+qX^{-j}))^{1/(2s)}\log (2X).$$
The conclusion of the theorem therefore follows on taking
$$s=(k-1)^2-(k-1)+1=k^2-3k+3,$$
for in such circumstances Corollary \ref{corollary1.2} delivers the bound
$$J_{s,k-1}(2X)\ll X^{2s-\frac{1}{2}k(k-1)+\eps}.$$
\end{proof}

The proof of \cite[Theorem 1.6]{Woo2012a} may be easily adapted to deliver estimates depending on 
common Diophantine approximations.

\begin{theorem}\label{theorem11.2} Let $k$ be an integer with $k\ge 4$, and let $\tau$ and $\del$ be
 real numbers with $\tau^{-1}>4(k^2-3k+3)$ and $\del>k\tau$. Suppose that $X$ is sufficiently large in
 terms of $k$, $\del$ and $\tau$, and further that $|f_k(\bfalp;X)|>X^{1-\tau}$. Then there exist integers
 $q$, $a_1,\ldots,a_k$ such that $1\le q\le X^\del$ and $|q\alp_j-a_j|\le X^{\del-j}$ $(1\le j\le k)$.
\end{theorem}

The proof of \cite[Theorem 1.7]{Woo2012a} likewise delivers the following result concerning the
 distribution modulo $1$ of polynomial sequences. Here, we write $\|\tet\|$ for 
$\underset{y\in \dbZ}{\min}|\tet-y|$.

\begin{theorem}\label{theorem11.3} Let $k$ be an integer with $k\ge 4$, and define $\tau(k)$ by
 $\tau(k)^{-1}=4(k^2-3k+3)$. Then whenever $\bfalp\in \dbR^k$ and $N$ is sufficiently large in terms
 of $k$ and $\eps$, one has
$$\min_{1\le n\le N}\|\alp_1n+\alp_2n^2+\ldots +\alp_kn^k\|<N^{\eps-\tau(k)}.$$
\end{theorem}

In each of Theorems \ref{theorem11.2} and \ref{theorem11.3}, the exponent $4(k^2-3k+3)$ represents 
an improvement on the exponent $4k(k-2)$ made available in \cite[Theorems 11.2 and 11.3]{Woo2013}. 
In \cite[Theorem 11.1]{Woo2013}, meanwhile, we established a conclusion similar to that of Theorem 
\ref{theorem11.1}, though with a weaker exponent $\sig(k)$ satisfying $\sig(k)^{-1}=2k(k-2)$. Our 
estimates supersede the Weyl exponent $\sig(k)=2^{1-k}$ when $k\ge 7$ (see 
\cite[Lemma 2.4]{Vau1997} and \cite[Theorem 5.1]{Bak1986}).\par

If one restricts to the situation where all coefficients save $\alp_k$ are zero, then further modest 
improvements may be obtained. When $\tet\in (0,k)$, let $\grm_\tet $ denote the set of real numbers
 $\alp$ having the property that, whenever $a\in \dbZ$ and $q\in \dbN$ satisfy $(a,q)=1$ and 
$|q\alp-a|\le P^{\tet-k}$, then one has $q>P^\tet$. The simplest improvements in earlier Weyl exponents 
stem from the following result of Boklan and Wooley \cite[Theorem 1.1]{BW2012}.

\begin{lemma}\label{Lemma11.4} Let $k\in \dbN$ with $k\ge 4$, and suppose that the exponent
 $\Del_{s,k-1}$ is permissible for $s\ge k$. Then for each $\eps>0$, one has
$$\sup_{\alp\in \grm_1}|g_k(\alp;X)|\ll X^{1-\sig(k)+\eps},$$
where
$$\sig(k)=\max_{s\ge k}\left( \frac{3-\Del_{s,k-1}}{6s+2}\right) .$$
\end{lemma}

By making use of the permissible exponents $\Del_{s,k}$ stemming from the discussion following 
(\ref{10.5}) one obtains the following conclusion by means of a na\"ive computer program.

\begin{theorem}\label{theorem11.5}
Suppose that $4\le k\le 20$ and that the positive numbers $\Sig_1(k)$ are defined as in Table $4$. Then 
for each $\eps>0$, one has
$$\sup_{\alp\in \grm_1}|g_k(\alp;X)|\ll X^{1-\sig(k)+\eps},$$
where $\sig(k)=1/\Sig_1(k)$.
\end{theorem}

$$\boxed{\begin{matrix} k&6&7&8&9&10\\
\Sig_1(k)&39.023&58.093&80.867&107.396&137.763\end{matrix}}$$
$$\boxed{\begin{matrix}k&11&12&13&14&15\\
\Sig_1(k)&172.027&210.222&252.370&298.487&348.580\end{matrix}}$$
$$\boxed{\begin{matrix}k&16&17&18&19&20\\
\Sig_1(k)&402.655&460.718&522.771&588.815&658.854\end{matrix}}$$
\vskip.2cm
\begin{center}\text{Table 4: Upper bounds for $\Sig_1(k)$ used in Theorem
 \ref{theorem11.5}.}\end{center}
\vskip.1cm
\noindent 

As we have noted, our estimates supersede the Weyl exponent $\sig(k)=2^{1-k}$ when $k\ge 7$. 
Heath-Brown \cite{HB1988} obtains the estimate
$$\sup_{\alp\in \grm_{3-\eps}}|g_k(\alp;X)|\ll X^{1-\sig(k)+\eps},$$
with $\sig(k)^{-1}=3\cdot 2^{k-3}$, an estimate that is superseded by Theorems \ref{theorem11.1} and 
\ref{theorem11.5} for $k\ge 8$. The work of Robert and Sargos \cite{RS2000} and of Parsell 
\cite{Par2013} yields sharper results subject to more restrictive Diophantine approximation hypotheses. 
These are also superseded by our conclusions for $k\ge 9$, though Parsell's work 
\cite[Theorem 1.2]{Par2013} shows that
$$\sup_{\alp\in \grm_{4-\eps}}|g_8(\alp;X)|\ll X^{1-\sig+\eps},$$
with $\sig^{-1}=80$, a conclusion slightly sharper than that implied by Theorem \ref{theorem11.5}, 
though under more restrictive hypotheses.\par

An asymptotic analysis of the argument establishing Theorem \ref{theorem11.5} shows that its 
conclusion holds in general with $\Sig_1(k)=2k^2-8k+O(1)$. However, for large values of $k$ one may 
derive a sharper bound by applying our earlier work \cite{Woo1995}, which we now recall.

\begin{lemma}\label{lemma11.6}
Let $R$ be an integer with $1\le R\le \tfrac{1}{2}k$, and write $\lam=1-R/k$. Suppose that $s$ and $t$ 
are positive integers with $s\ge \tfrac{1}{2}k(k-1)$, and suppose further that the exponents 
$\Del_{s,k-1}$ and $\Del_{t,k}$ are permissible. Then we have
$$\sup_{\alp_k\in \grm_\lam}|f_k(\bfalp;X)|\ll P^\eps (P^{1-\mu(k)}+P^{1-\nu(k)}),$$
where
$$\mu(k)=\frac{R-\Del_{s,k-1}}{2Rs}\quad \text{and}\quad \nu(k)=\frac{k-R(1+\Del_{t,k})}{2tk}.$$
\end{lemma}

\begin{proof} This is \cite[Theorem 2]{Woo1995}.
\end{proof}

\begin{theorem}\label{theorem11.7}
Suppose that $k$ is a large positive integer. Then for each $\eps>0$, one has
$$\sup_{\alp\in \grm_1}|g_k(\alp;X)|\ll X^{1-\sig(k)+\eps},$$
where
$$\sig(k)^{-1}=2k^2-\frac{1}{\sqrt{3}}k^{3/2}+O(k).$$
\end{theorem}

\begin{proof} We apply Lemma \ref{lemma11.6} with
$$s=(k-1)^2-r(k-1)+\tfrac{1}{2}r(r+3)-1$$
and
$$t=k^2-uk+\tfrac{1}{2}u(u+3)-1,$$
with the permissible exponents $\Del_{s,k-1}$ and $\Del_{t,k}$ determined via Theorem 
\ref{theorem1.1}. With a little experimentation, one finds that the optimal choices of the parameters $r$, 
$u$ and $R$ in the application of Lemma \ref{lemma11.6} are all of order $\sqrt k$. We therefore put 
$R=\lfloor \tet \sqrt{k}\rfloor$, $r=\lfloor \phi \sqrt{k}\rfloor$ and $u=\lfloor \psi\sqrt{k}\rfloor$, with 
$\tet$, $\phi$ and $\psi$ positive parameters to be chosen in due course. One finds from Theorem 
\ref{theorem1.1} that one has permissible exponents $\Del_{s,k-1}$ and $\Del_{t,k}$, with
$$\Del_{s,k-1}=\tfrac{1}{2}\phi^2+O(k^{-1/2})\quad \text{and}\quad 
\Del_{t,k}=\tfrac{1}{2}\psi^2+O(k^{-1/2}).$$
In addition, one has
$$s=k^2(1-\phi k^{-1/2}+O(1/k))\quad \text{and}\quad t=k^2(1-\psi k^{-1/2}+O(1/k)).$$
Thus, in our application of Lemma \ref{lemma11.6}, we obtain
\begin{align*}
2k^2\mu(k)&=(1-\tfrac{1}{2}\phi^2\tet^{-1}k^{-1/2})(1+\phi k^{-1/2})+O(1/k)\\
&=1+\tfrac{1}{2}(2\phi-\phi^2\tet^{-1})k^{-1/2}+O(1/k),
\end{align*}
and
\begin{align*}
2k^2\nu(k)&=(1-(1+\tfrac{1}{2}\psi^2)\tet k^{-1/2})(1+\psi k^{-1/2})+O(1/k)\\
&=1+\tfrac{1}{2}(2\psi-(2+\psi^2)\tet )k^{-1/2}+O(1/k).
\end{align*}
A rapid optimisation reveals that we should take $\phi=\tet$ and $\psi=\tet^{-1}$ in order to optimise 
these two expressions, and then the optimal choice for $\tet$ is determined by the equation 
$\tet=\tet^{-1}-2\tet$. Thus we deduce that one should take $\tet=1/\sqrt{3}$, $\phi=1/\sqrt{3}$ and 
$\psi=\sqrt{3}$, delivering the exponents
$$\mu(k)^{-1}=2k^2(1-1/(2\sqrt{3k})+O(1/k))$$
and
$$\nu(k)^{-1}=2k^2(1-1/(2\sqrt{3k})+O(1/k)).$$
The conclusion of the theorem now follows at once from Lemma \ref{lemma11.6}.
\end{proof}

It would appear that the exponents provided by means of Lemma \ref{lemma11.6} do not supersede 
those provided by Theorem \ref{theorem11.5}, by reference to Table 4, in the range $6\le k\le 20$.

\section{Further applications} We turn next to Tarry's problem. When $h$, $k$ and $s$ are positive 
integers with $h\ge 2$, consider the Diophantine system
\begin{equation}\label{11.1}
\sum_{i=1}^sx_{i1}^j=\sum_{i=1}^sx_{i2}^j=\ldots =\sum_{i=1}^sx_{ih}^j\quad (1\le j\le k).
\end{equation}
Let $W(k,h)$ denote the least natural number $s$ having the property that the simultaneous equations 
(\ref{11.1}) possess an integral solution $\bfx$ with
$$\sum_{i=1}^sx_{iu}^{k+1}\ne \sum_{i=1}^sx_{iv}^{k+1}\quad (1\le u<v\le h).$$

\begin{theorem}\label{theorem12.1} When $h$ and $k$ are natural numbers with $h\ge 2$ and 
$k\ge 3$, one has $W(k,h)\le \tfrac{5}{8}(k+1)^2$.\end{theorem}

\begin{proof} The argument of the proof of \cite[Theorem 1.3]{Woo2012a} shows that $W(k,h)\le s$ 
whenever one can establish the existence of a permissible exponent $\Del_{s,k+1}$ with 
$\Del_{s,k+1}<k+1$. But by taking $r=\lfloor \tfrac{1}{2}(k+1)\rfloor$ in Theorem \ref{theorem1.1}, 
one finds that whenever $s\ge \tfrac{5}{8}(k+1)^2$, one has $\Del_{s,k+1}<\tfrac{2}{15}k$. The 
conclusion of the theorem follows immediately.
\end{proof}

In \cite[Theorem 11.4]{Woo2013}, we obtained the weaker bound
$$W(k,h)\le k^2-\sqrt{2}k^{3/2}+4k.$$
It is plain that there is plenty of room to spare in the above proof of Theorem \ref{theorem12.1}. 
This is a topic we intend to pursue elsewhere.\par

We note also that on writing
$$\grS(s,k)=\sum_{q=1}^\infty 
\underset{(a_1,\ldots ,a_k,q)=1}{\sum_{a_1=1}^q\dots \sum_{a_k=1}^q}
\Bigl|q^{-1}\sum_{r=1}^qe((a_1r+\ldots +a_kr^k)/q)\Bigr|^{2s}$$
and
$$\calJ(s,k)=\int_{\dbR^k}\Bigl| \int_0^1e(\bet_1\gam+\ldots +\bet_k\gam^k)\d\gam 
\Bigr|^{2s}\d\bfbet ,$$
the method of proof of \cite[Theorem 1.2]{Woo2012a} may be modified in the light of Corollary 
\ref{corollary1.2} to obtain the following conclusion.

\begin{theorem}\label{theorem12.2}
Suppose that $k\ge 3$ and $s\ge k^2-k+2$. Then one has the asymptotic formula
$$J_{s,k}(X)\sim \grS(s,k)\calJ(s,k)X^{2s-\frac{1}{2}k(k+1)}.$$
\end{theorem}

In \cite[\S11]{Woo2013}, such a conclusion was obtained for $s\ge k^2$. A similar improvement holds 
also for work on the asymptotic formula in the Hilbert-Kamke problem.\par

Finally, write
$$F_k(\bfbet;X)=\sum_{1\le x\le X}e(\bet_kx^k+\bet_{k-2}x^{k-2}+\ldots +\bet_1x).$$
L.-K. Hua \cite{Hua1965} investigated the problem of bounding the least integer $C_k$ such that, 
whenever $s\ge C_k$, one has
$$\oint |f_k(\bfalp;X)|^s\d\bfalp \ll X^{s-\frac{1}{2}k(k+1)+\eps},$$
and likewise the least integer $S_k$ such that, whenever $s\ge S_k$, one has
\begin{equation}\label{12.2a}
\oint |F_k(\bfbet;X)|^s\d\bfbet \ll X^{s-\frac{1}{2}(k^2-k+2)+\eps}.
\end{equation}
We are able to reduce the upper bounds for $C_k$ and $S_k$ provided by Hua \cite{Hua1965}, and also 
the subsequent improved bounds given in our earlier work \cite{Woo2012a, Woo2013}.

\begin{theorem}\label{theorem12.3} When $k\ge 3$, one has $C_k\le 2k^2-2k+2$. Meanwhile, one has
$$S_4\le 22,\quad S_5\le 34,\quad S_6\le 52,\quad S_7\le 66,\quad S_8\le 88$$
and $S_k\le 2k^2-6k+6$ for $k\ge 9$.
\end{theorem}

\begin{proof} The bound on $C_k$ is immediate from Corollary \ref{corollary1.2} via (\ref{2.3}) and 
orthogonality. In order to establish the bound on $S_k$, we begin by observing that 
\cite[equation (10.10)]{Woo2012a} supplies the estimate
\begin{equation}\label{12.2}
\oint |F_k(\bfbet;X)|^{2t}\d\bfbet \ll X^{k-2+\eps}J_{t,k}(2X)+X^{\eps-1}J_{t,k-1}(2X).
\end{equation}
It follows from Corollary \ref{corollary1.2} that when $t\ge k^2-3k+3$, one has
\begin{equation}\label{12.3}
J_{t,k-1}(2X)\ll X^{2t-\frac{1}{2}k(k-1)+\eps}.
\end{equation}
In addition, explicit computations of the type described following (\ref{10.5}) show that the exponent 
$\Del_{t,k}=1$ is permissible whenever $t\ge t^*(k)$ and $4\le k\le 8$, where
$$t^*(4)=11,\quad t^*(5)=17,\quad t^*(6)=26,\quad t^*(7)=33,\quad t^*(8)=44.$$
For these exponents, therefore, when $t\ge t^*(k)$, one has the upper bound
\begin{equation}\label{12.4}
J_{t,k}(2X)\ll X^{2t-\frac{1}{2}k(k+1)+1+\eps}.
\end{equation}
By substituting (\ref{12.3}) and (\ref{12.4}) into (\ref{12.2}), we obtain the desired conclusion 
(\ref{12.2a}) with $s=2t$ for $4\le k\le 8$.\par

When $k\ge 9$, we instead apply Theorem \ref{theorem1.1} with $r=5$ to show that whenever
$t\ge k^2-5k+19$, the permissible exponent $\Del_{t,k}$ is permissible, where
$$\Del_{t,k}=\frac{10(k-5)}{k^2-6k+20}<1.$$
When $k\ge 9$ and $t\ge k^2-3k+3$, therefore, we again have the estimates (\ref{12.3}) and 
(\ref{12.4}), and the estimate (\ref{12.2a}) with $s=2t$ follows just as before. This completes the proof
 of the theorem.
\end{proof}

For comparison, in \cite[Theorem 11.6]{Woo2013} we derived the somewhat weaker bounds 
$C_k\le 2k^2-2$ and $S_k\le 2k^2-2k$.

\bibliographystyle{amsbracket}
\providecommand{\bysame}{\leavevmode\hbox to3em{\hrulefill}\thinspace}

\end{document}